\documentclass[english, 12pt]{amsart}
\usepackage{amssymb}
\usepackage{amsfonts}
\usepackage{amscd}
\usepackage{color}
\usepackage{tikz}

\providecommand{\noopsort}[1]{}

\usepackage{hyperref}
\usepackage{cleveref}

\newbox\removebox
\newcommand\remove[2][blue]{%
\setbox\removebox=\ifmmode\hbox{$#2$}\else\hbox{#2}\fi%
\leavevmode
\rlap{\textcolor{#1}{\vrule height0.8ex depth-0.5ex width\wd\removebox}}%
\box\removebox
}
\long\def\bigremove#1{%
\par\setbox\removebox=\vbox{#1}%
\vbox{%
\vbox to0pt{\hbox{\tikz\draw[color=blue,thick] (0,0) -- (\wd\removebox,-\ht\removebox)  (\wd\removebox,0) -- (0,-\ht\removebox);}}
\box\removebox
}
}

\usepackage{mathrsfs} 

\newcommand{\cCexp}{\cC^{\mathrm{exp}}}
\newcommand{\cCe}{\cC^{\mathrm{e}}}

\newcommand{\cQe}{\cQ^{\mathrm{e}}}
\newcommand{\cQexp}{\cQ^{\mathrm{exp}}}

\newcommand{\rad}{\operatorname{rad}}

\newcommand{\new}[1][]{_{\mathrm{new}#1}}

\def\VF{\mathrm{VF}}

\def\VG{\mathrm{VG}}
\newcommand{\RF}{{\rm RF}}

\def\ac{{
{\rm ac}}}

\def\RDef{\operatorname{RDef}}

\def\LPas{\cL_{\rm DP}}

\def\gTPas
{{\rm gDP}}

\def\TPres
{{\rm PRES}}

\def\Hen{\cS} 

\def\11{{\mathbf 1}}
\def\AA{{\mathbb A}}

\def\LL{{\mathbb L}}

\def\NN{{\mathbb N}}

\def\ZZ{{\mathbb Z}}

\def\cC{{\mathscr C}}

\def\cL{{\mathcal L}}

\def\cP{{\mathcal P}}
\def\cQ{{\mathcal Q}}
\def\cR{{\mathscr R}}
\def\cS{{\mathcal S}}

\def\llp{\mathopen{(\!(}}

\def\rrp{\mathopen{)\!)}}

\newtheorem{mainthm}
{Theorem}

\newtheorem{cor}[subsubsection]{Corollary}
\newtheorem{prop}[subsection]{Proposition}

\theoremstyle{definition}
\newtheorem{defn}[subsection]{Definition}
\newtheorem{defnsubsub}[subsubsection]{Definition}
\newtheorem{example}[subsubsection]{Example}

\newtheorem{def-prop}[subsubsection]{Proposition-Definition}
\newtheorem{def-thm}[subsubsection]{Theorem-Definition}
\newtheorem{def-lem}[subsubsection]{Lemma-Definition}

\newtheorem{convention}[subsubsection]
{Convention}

\theoremstyle{remark}
\newtheorem{remark}[subsection]
{Remark}
\newtheorem{remarksubsub}[subsubsection]
{Remark}

\Crefname{prop}{Proposition}{Propositions}
\Crefname{thm}{Theorem}{Theorems}
\Crefname{lem}{Lemma}{Lemmas}
\Crefname{cor}{Corollary}{Corollaries}
\Crefname{conj}{Conjecture}{Conjectures}
\Crefname{defn}{Definition}{Definitions}
\Crefname{defnsubsub}{Definition}{Definitions}

\theoremstyle{plain}

  {\par\medskip\noindent #1\par\begingroup%
    \advance\leftskip by 1em\advance\rightskip by 1em}%
  {\par\endgroup}

\newcommand{\ord}{\operatorname{ord}}

\usepackage[T1]{fontenc}
\usepackage{mathrsfs}

\definecolor{immi's color}{rgb}{.9,.3,0}

\begin{document}

\setcounter{tocdepth}{1} 

\author
{Raf Cluckers}

\address{Univ.~Lille,  CNRS, UMR 8524 - Laboratoire Paul Painlev\'e, F-59000 Lille, France, and
KU Leuven, Department of Mathematics, B-3001 Leu\-ven, Bel\-gium}
\email{Raf.Cluckers@univ-lille.fr}
\urladdr{http://rcluckers.perso.math.cnrs.fr/}

\author
{Immanuel Halupczok}
\address{Mathematisches Institut,
Universit\"atsstr. 1, 40225 D\"usseldorf,
Germany}
\email{math@karimmi.de}
\urladdr{http://www.immi.karimmi.de/en/}


\thanks{The authors thank Michel Raibaut, Jorge Cely, Fran\c cois Loeser, Thomas Scanlon, Julien Sebag and an anonymous referee for useful feedback on the paper.
R.C. was partially supported by the European Research Council under the European Community's Seventh Framework Programme (FP7/2007-2013) with ERC Grant Agreement nr. 615722
MOTMELSUM and KU Leuven IF C14/17/083, and thanks the Labex CEMPI  (ANR-11-LABX-0007-01). The second author is partially supported by the research training group
\emph{GRK 2240: Algebro-Geometric Methods in Algebra, Arithmetic and Topology},
and by the individual research grant \emph{Archimedische und nicht-archimedische Stratifizierungen h\"oherer Ordnung}, both funded by the DFG}


\title[Evaluation  of motivic functions, integrability]{Evaluation of motivic functions, non-nullity, and integrability in fibers}

\begin{abstract}
We define an operation of evaluation at a point for motivic constructible (exponential) functions from the Cluckers-Loeser framework of motivic integration and show
that two such motivic functions are abstractly equal if and only if their evaluations at each point are the same.
We similarly characterise relative integrability  in terms of integrability in each fiber separately. These results simplify the mentioned frameworks of motivic integration and their usage.
\end{abstract}

\maketitle

\section{Introduction}

Motivic integration is inspired by $p$-adic integration, but by the lack of topological local compactness, it is not subject to standard measure theory. Instead, different, more abstract formalisms have been developed: by Kontsevich initially in a smooth setting \cite{Konts}; by Batyrev \cite{Batyrev} based on $\sigma$-algebras; variants by Denef and Loeser more generally on singular spaces, both geometric \cite{DLinvent} and arithmetic-geometric \cite{DL}; by Loeser, Nicaise and Sebag \cite{LoeserSeb, NicaSeba} in mixed and positive characteristic using N\'eron models; and, more recently, with Fubini results and Fourier transform included in the realm of motivic integration by Loeser and the first author \cite{CLoesI} -- \cite{CLexp}
and, by Hrushovski and Kazhdan \cite{HK}. Motivic integration has found striking applications going from the equality of Hodge numbers for birational Calabi-Yau variaties, to the study of  stringy invariants \cite{Batyrev,Veys:stringy,Baty-Moreau}, motivic Milnor fibers \cite{NicaSeba, HruLoeser:mono} and the log-canonical threshold \cite{MustJAMS}, and, to applications in the Langlands program \cite{CHL,YGordon,CGH2,Casselman-Cely-Hales}.

In this paper, we develop new insights in the treatment of \cite{CLoes,CLexp}.
\emph{A priori}, a motivic function in this formalism is an abstract geometric object. We show that those objects
can be considered as actual functions, in the sense that
they are determined by evaluation in points (Theorem \ref{thm:no:null:intro}). Similarly,
(abstract) relative integrability is determined fiber by fiber
(Theorem \ref{thm:integrability-intro}).  Both these insights lead to simplifications of the theory of motivic integration from \cite{CLoes,CLexp}, e.g. as given in Corollaries~\ref{cor:eval:hom:pull}, \ref{cor:eval:push}, \ref{cor:no:null:main} and \ref{cor:CelyRaib}.

\bigskip

To put this paper into a historic perspective, recall that
motivic constructible functions (of class $\cC$) and their integrals from \cite[Theorem 10.1.1 and Sections 13.2, 14.3]{CLoes} are an abstract framework of integration with Fubini properties, inspired by the $p$-adic functions introduced by Denef in \cite{Denef3}, see \cite[Section 1.5]{Denef1}.
These were extended both $p$-adically and motivically with additive characters and Fourier transforms in \cite{CLexp} (forming the class $\cCexp$) and uniformly through all $p$-adic fields in \cite{CHallp}. This went along with a growing understanding of the model theory of valued fields, in particular with cell decomposition results from \cite{Denef2,Pas} and later variants like \cite[Theorem 5.3.1]{CHallp} for all $p$-adic fields.
We hope that our work will be helpful for going from
uniform $p$-adic results in the style of  e.g.~\cite{CGH5,CHLR} to deeper motivic results as in \cite{Raib-motivic}.
The non-nullity results of this paper may also play a role in the study of motivic analogues of the Fundamental Lemma of the Langlands program as in \cite{LoesWyss}, and in a motivic variant of the $p$-adic results from \cite{CH:GrQp} on Kontsevich-Zagier style integral transformation rules.

\subsection{}
Let us sketch our main results in some more detail. For any motivic constructible (exponential) function $f$ as defined in \cite[Section 5.3]{CLoes}, resp.~\cite[Section 3.3]{CLexp} and recalled in Section \ref{sec:defn-proof}, we define the notion of the evaluation $f(x)$ of $f$ at a point $x$ in the domain $X$ of $f$ (see Definition \ref{defn:eval}), and prove the following non-nullity result. (A discussion of the terminology follows after Theorem \ref{thm:no:null:intro}.)



\begin{mainthm}
\label{thm:no:null:intro}
Let $X$ be an $\Hen$-definable 
set and let $f$ and $g$ be in $\cC(X)$. 
Then the following statements are equivalent.
\begin{enumerate}
\item\label{item:1:1}  One has  $f(x)=g(x)$ for all points $x$ on $X$.
\item\label{item:1:2}  One has $f=g$ in $\cC(X)$.
\end{enumerate}
Moreover, the same equivalence holds with $\cCe(X)$
and with $\cCexp(X)$ instead of $\cC(X)$.
In particular, one has that $f(x)=0$ for all points $x$ on $X$ if and only if $f=0$.
\end{mainthm}
We now briefly explain the terminology used in Theorem \ref{thm:no:null:intro}: 
\begin{itemize}
 \item
The notion of $\Hen$-definable set is fixed in \Cref{defn:S} using the more general terminology of Section \ref{sec:def-set}, which is designed to deal uniformly with non-elementary classes $\cS$ of structures and with (points on) definable sets within those structures.
It corresponds to a certain notion of subassignments of 
\cite{CLoes}, see Remark \ref{rem:T-sub}. Apart from being shorter, it is also more handy for our use, especially since we need to change $\cS$ flexibly.
The relevance of non-elementary classes in this context is explained in Section \ref{sec:non-el}.

\item The rings $\cC(X)$, $\cCe(X)$ and $\cCexp(X)$ are recalled in Sections \ref{sec:C} and \ref{sec:Cexp}, from \cite[Section 5.3]{CLoes} and \cite[Sections 3.3 and 6.2]{CLexp}. Elements of $\cC(X)$ are called motivic constructible functions on $X$ in \cite{CLoes}. The ring $\cCexp(X)$ additionally contains a `motivic' additive character to enable motivic Fourier transform, and its elements are called motivic constructible exponential functions in \cite{CLexp}. The ring $\cCe(X)$ lies in between $\cC(X)$ and $\cCexp(X)$, with a motivic additive character on the residue field only.
\item The notion of a point $x$ on an $\Hen$-definable set $X$ of Definition \ref{defn:points} is similar to the one of \cite[Section 2.6]{CLoes}; essentially, it is an element of $X_K$ for some structure $K$  in $ \Hen$, see \Cref{defn:S-def}. The precise definition of $f(x)$ for a point $x$ on $X$ is given in Definition \ref{defn:eval}.
\end{itemize}

The ring $\cC(X)$ of motivic constructible functions on an $\Hen$-definable set $X$ is defined abstractly and not so much is known about it in general. The same holds for its variants $\cCe(X)$ and $\cCexp(X)$. The point of Theorem \ref{thm:no:null:intro} is that it permits one to view an element $f$ of $\cC(X)$ (and of $\cCe(X)$ and $\cCexp(X)$) 
as an actual function. In other words, knowing $f$ abstractly or knowing the function sending points $x$ on $X$ to $f(x)$ amounts to the same information. This leads to various possible simplifications of the presentation and usage of the Cluckers--Loeser framework of motivic integration.
For example, the delicate projection formula for motivic integrals shown recently in \cite[Theorem 1.1]{CelyRaib} (which refines the projection formulas from \cite[Theorem 10.1.1  A3, Proposition 13.2.1 (2)]{CLoes} and \cite[Theorem 4.1.1 A3]{CLexp}), follows easily from Theorem \ref{thm:no:null:intro};
see Corollary~\ref{cor:CelyRaib}.
As another example of a simplification,
using Theorem \ref{thm:no:null:intro}, one can now consider the abstractly defined pull-back $g^*(f)$ of an element $f$ of $\cC(Z)$ under an $\Hen$-definable function $g:X\to Z$ as an actual composition
(see Corollary \ref{cor:eval:hom:pull}). Note that, originally, $g^*(f)$ is defined more abstractly in $\cC(X)$, and similarly for the cases of $\cCe$ and $\cCexp$. Likewise, the relative motivic integral (along $g$) of a suitably integrable function $h$ in $\cC(X)$
 is determined by the motivic integral of $h$ restricted to each fiber of $g$ (see Definitions \ref{def:integrab}, \ref{def:integral} and Corollary \ref{cor:eval:push}).

\bigskip

The next theorem completes our main results. It states that checking relative integrability can also be done in a point-wise way.
This allows us to reprove the integrability condition obtained in \cite[Theorem 1.1 (1)]{CelyRaib} (see  (1) of Corollary~\ref{cor:CelyRaib} below).
The precise meaning of integrability in dimension $d$ comes from \cite[Section 14]{CLoes} and \cite[Section 4.3]{CLexp} and will be specified below in Definition \ref{def:integrab}.

Natural definitions of $g^{-1}(z)$ and $f_{|g^{-1}(z)}$ for $g:X\to Z$ an $\Hen$-definable function, $z$ a point on $Z$, and $f$ in $\cC(X)$ (or in $\cCe(X)$, or $\cCexp(X)$) are given In Section \ref{sec:eval}, again using general notions from Section \ref{sec:def-set} together with \Cref{defn:S} for $\cS$.

\begin{mainthm}[Relative integrability versus integrability in all fibers]\label{thm:integrability-intro}
Let $X$ and $Z$ be $\Hen$-definable sets, let $f$ be in $\cC(X)$, and let $g:X\to Z$ be an $\Hen$-definable function such that the fibers have (valued field) dimension $\leq d$ for some integer $d\geq 0$ (see \Cref{defndim}).
Then the following are equivalent, with integrability as defined in Definition \ref{def:integrab}. 
\begin{enumerate}
\item $f$ is integrable in relative dimension $d$ over $Z$ along $g:X\to Z$ (that is, in the fibers of $g$).
\item For each point $z$ on $Z$, the restriction $f_{|g^{-1}(z)}$ is integrable in dimension $d$.
\end{enumerate}
The same equivalence holds when $f$ lies in $\cCe(X)$ or in $\cCexp(X)$.
\end{mainthm}

The notion of integrability from \Cref{def:integrab} is indirectly based on classical summability of series of positive real numbers, indexed by the value group, as is the case for the integrability notions in \cite{CLoes} and \cite{CLexp}. This uses that the value group of the fields in $\Hen$ is $\ZZ$ itself and not a nonstandard model of (the theory of) $\ZZ$.
In Section \ref{sec:int:revisited}, we give a simplified account of the integrability notions from \cite[Theorem 10.1.1 and Sections 13.2, 14.3]{CLoes} and \cite[Theorem 4.1.1]{CLexp}, based on direct criteria rather than on uniqueness and existence results.
The definitions of relative integrability from \cite[Theorem 10.1.1 and Sections 13.2, 14.3]{CLoes} and \cite[Theorem 4.1.1]{CLexp} thus become more natural and intuitive by Theorem \ref{thm:integrability-intro} and by the criteria from Section \ref{sec:int:revisited}. Note that integrating in fixed relative dimension $d$ as in \Cref{thm:integrability-intro} is key to general relative integration, as other relative dimensions can be reduced to this case up to taking a finite partition of $X$, see the discussion just below \Cref{def:integrab}.

\subsection{}\label{sec:non-el}
We now discuss the use and subtlety of non-elementary classes of valued fields.
\begin{itemize}
\item
In the context of motivic integration, one often uses non-elementary classes of valued fields. This is mainly to fix the value group to be $\ZZ$ and to use standard summation on it. In particular, geometric power series and their derivatives like $\sum_{i\ge 0} i^nr^{i}$ with real $r<1$ and integer $n\ge 0$ are classical when $i$ runs over $\ZZ$, but they would become unclear if $i$ were to live in a more abstract value group. Such concrete summation allows us to simplify sums over the value group, in comparison with e.g.~\cite{HK}, where other value groups are used.

\item One of the subtleties for proving Theorem \ref{thm:no:null:intro} is that the structures in $\Hen$ do not form an elementary class, so that logical compactness can not be used, and neither can infinitely large elements be used. This is different from the Hrushovski-Kazhdan framework of motivic integration \cite{HK}, where an elementary class of models is used throughout, and where non-nullity as given by Theorem \ref{thm:no:null:intro} is not an issue. For example, for a relation like $\LL^i[X] = [Y]$ in $\cC(\ZZ)$ where $i$ is the variable, running over the value group, $\LL$ is the class of the affine line over the residue field, and with $X$ and $Y$ some $\Hen$-definable sets in the residue field sort,  this makes a difference. Knowing $\LL^i[X] = [Y]$ in $\cC(\{i\})$ for each integer $i$ separately is very different than it would be to know it for each element $i$ of an elementary class of value groups, or, even just for a single infinitely large value of $i$.\footnote{This is a hypothetical example since the framework with $\cC$ does not yet exist in a variant with more general value groups.}
\end{itemize}

\subsection{}

Both Theorems \ref{thm:no:null:intro} and \ref{thm:integrability-intro} can easily be generalized to other classes of motivic functions than $\cC$, $\cCe$ and $\cCexp$ from  \cite[Section 5.3]{CLoes} and \cite[Sections 3.3 and 6.2]{CLexp}, like the ones listed in \cite[Section 3.1]{CLbounded} and corresponding variants $\cC^{\rm rat}$ (with more localisations) obtained from groupifying the semigroups $\cC^{\rm rat}_+$ from \cite[ Definition 3.1.6]{Kien:rational}. Furthermore, Theorem \ref{thm:no:null:intro} can be used to simplify the frameworks of \cite{CLoes,CLexp,CLbounded,Kien:rational} and to simplify issues about null-functions and parameter integrals from e.g.~\cite{Casselman-Cely-Hales,CelyRaib,CHLR,Raib-motivic}.

\subsection{Description of the ideas}\label{sec:desc}


Our notion of evaluation and the related notions are natural and based on changing theories by adding constant symbols and passing to complete theories (see Section \ref{sec:eval}).  Adding constants and changing theories accordingly is standard in model theory, and is often done similarly, e.g.~for working with types. We develop some general terminology in Section \ref{sec:def-set}, where we allow non-elementary classes of structures to work with uniformly definable sets.

The proof of Theorem \ref{thm:no:null:intro} for $\cC$ and $\cCe$ is reduced to the case where $X$ does not involve valued field coordinates and thus only residue field and value group coordinates. This reduction is done using Proposition \ref{prop:twisted:boxes} which relies on work of Pas \cite{Pas}. The case for $\cCexp$ is reduced to $\cCe$ by Proposition \ref{lem:exp-to-con}.  On the residue field level, one uses logical compactness to show non-nullity (see Case 1 of the proof of Theorem \ref{thm:no:null:intro}). This is allowed since we impose in \Cref{defn:S} that the residue fields of the structures in $\Hen$ run over all models of a theory and thus form an elementary class.  For the value group (which is $\ZZ$) a more delicate reasoning is needed.
Essentially, we use basic tricks like taking the difference of $f(i+1)$ with $f(i)$, which makes certain degrees (in $i$ and in $\LL^i$) go down so that induction comes to help (see Case 2 of the proof of Theorem \ref{thm:no:null:intro}).

For the proof of Theorem \ref{thm:integrability-intro}, we again reduce from $\cCexp$ to $\cCe$ using Proposition \ref{lem:exp-to-con}.  Next, we reduce to the case that $X$ does not involve valued field coordinates by Propositions \ref{prop:twisted:boxes} and \ref{prop:VF-RF},
so that the occurring dimension (in the valued field) is  $d=0$. By the proofs of \cite[Theorem 10.1.1]{CLoes} and \cite[Theorem 4.1.1]{CLexp}, integrability conditions always come from summation over the value group via summability conditions of real numbers, by replacing $\LL$ (the class of the affine line in the residue field sort) by real numbers $q>1$. This allows us to use a criterion about summation of motivic functions of class $\cC$ and $\cCe$ over $\ZZ^R$  (see Proposition \ref{prop:crit}), and then the proof of Theorem \ref{thm:integrability-intro} can be finished by an application of Theorem \ref{thm:no:null:intro}.

Familiarity with \cite{CLoes,CLexp} is necessary to understand full details; nevertheless, a more global reading of this paper is also possible without such familiarity.


\section{Definable sets within collections of structures} 
\label{sec:def-set}
This section is completely general, and introduces handy terminology in first order logic to work simultaneously with all structures in a given (not necessarily elementary) class. From Section \ref{sec:defn-proof} on, we will put ourselves back in the context of henselian valued fields.

In this section we define a variant of the notion of `definable  set'  which specifies at the same time in which structures they live. Classically in model theory one considers either a fixed structure, or, one works with an elementary class of structures. We allow non-elementary classes of structures, and we make that explicit in our notation.
In particular, we will speak of $\cS$-definable rather than $\cL$-definable with $\cS$ a collection of $\cL$-structures for some language $\cL$.
While some related objects and manipulations could be defined in terms of the common theory of the structures $\cS$, others, like the notion of ``points on definable sets'' from \Cref{defn:points} depend finely on the collection $\Hen$ itself.

In this section, let  $\Hen$ be a nonempty collection of $\cL$-structures in some language $\cL$ (a first order language, as usual in model theory). We will, at the beginning of Section \ref{sec:defn-proof}, fix a language $\cL$ and a collection $\cS$ of $\cL$-structures according to \Cref{defn:S}.

\begin{defn}\label{defn:S-def}
By an $\Hen$-definable set is meant a collection of sets $X = (X_K)_{K \in \Hen}$ such that there exists an $\cL$-formula $\varphi$ with $\varphi(K) = X_K$ for all $K$ in $\Hen$ and where $\varphi(K)$ stands for the set of tuples in the  $\cL$-structure $K$ that satisfy $\varphi$.

By an $\Hen$-definable function $g:X\to Z$ between $\Hen$-definable sets $X$ and $Z$ is meant a collection of functions $g_K:X_K\to Z_K$ for $K$ in $\Hen$ such that the collection of the graphs of the $g_K$ forms an $\Hen$-definable set.
\end{defn}
Basic terminology of set theory can and will be used: An $\Hen$-definable bijection between $\Hen$-definable sets $X$ and $Z$ is an  $\Hen$-definable function $g:X\to Z$ such that $g_K$ is a bijection for each $K$ in $\Hen$, and similarly for $\Hen$-definable injections, surjections, subsets, Cartesian products, pre-images, fiber products,  etcetera. Slightly more subtly, we call a collection of $\cS$-definable subsets $X_i$ of $X$ a partition of $X$, if for each $K$ in $\cS$ the nonempty sets among the $X_{i,K}$ form a partition of $X_K$, and, for each $i$ there is at least one structure $K$ in $\cS$ such that $X_{i,K}$ is nonempty.

\begin{defn}\label{defn:points}
For an $\Hen$-definable set $X$, a pair $x=(x_1,K_1)$ with $x_1$ in $X_{K_1}$ and  $K_1$ in $\Hen$ is called a point on 
$X$.
\end{defn}

\begin{defn}\label{defn:L(x)}
Given a point $x=(x_1,K_1)$ on an $\Hen$-definable set $X$, write $\cL(x)$ for the language $\cL$ expanded by constant symbols for the entries of the tuple $x_1$, and write $\Hen(x)$ for the collection of $\cL(x)$-structures which are $\cL(x)$-expansions of $\cL$-structures in $\Hen$ that  are elementarily equivalent to the $\cL(x)$-structure $K_1$ (where in $K_1$, the new constants are interpreted by $x_1$).
\end{defn}

\begin{defn}\label{defn:x}
Let $x=(x_1,K_1)$ be a point on an $\Hen$-definable set. By a harmless abuse of notation, we denote by $\{x\}$ the $\Hen(x)$-definable set sending $K$ in $\Hen(x)$ to  $\{x\}_K := \{x_{1,K}\}$, where $x_{1,K}$ is the interpretation in the $\cL(x)$-structure $K$ of the tuple of constant symbols introduced for $x_1$.
Likewise, if $g:X\to Z$ is an $\Hen$-definable function and $x=(x_1,K_1)$ is a point on $X$, then we write $g(x)$ for the corresponding point $(z_1,K_1)$ on $Z$, with
$z_1=g_{K_1}(x_1)$.
(With a similarly harmless abuse of notation, one sometimes just writes $x$ instead of $\{x\}$.)
\end{defn}

Let us fix some further notation. 
Consider a point $x$ on an $\Hen$-definable set $X$. Then any $\Hen$-definable set $Z$
naturally determines an $\Hen(x)$-definable set, which we denote by $Z_{\Hen(x)}$. In detail, $Z_{\Hen(x)}$ associates to $K$ in $\Hen(x)$ the set $Z_{K_{|\cL}}$, where $K_{|\cL}$ in $\Hen$ is the $\cL$-reduct of $K$. We call $Z_{\Hen(x)}$ the base change of $Z$ to $\Hen(x)$. (Note that the notation $Z_{\Hen(x)}$ is in close analogy to the usual notation for base change in algebraic geometry.)
Similarly, any $\Hen$-definable function $g:Z\to Y$ between $\Hen$-definable sets determines an $\Hen(x)$-definable function $g_{\Hen(x)}:Z_{\Hen(x)}\to Y_{\Hen(x)}$.
We use natural related notation, for example, for a point $y$ on $Y$, we write  $g^{-1}(y)$ for the obvious $\Hen(y)$-definable subset of $Z_{\Hen(y)}$, that is, for the pre-image of $\{y\}$ under $g_{\Hen(y)}$
(which could more formally also be denoted by $(g_{\Hen(y)})^{-1}(\{y\})$).  

When $p:W\subset X\times Y \to Y$ is the projection to the $Y$-coordinate with $X,Y,W$ some $\Hen$-definable sets,  we sometimes write $W_y$ for
$p^{-1}(y)$ (which formally is an $\Hen(y)$-definable subset of $W_{\Hen(y)}$).
If furthermore $h:W\to V$ is an $\Hen$-definable function, then we write $h(\cdot,y)$ for the restriction of $h_{\Hen(y)}$ to $W_y$.

\section{Definition of evaluation and corollaries of the main results}\label{sec:defn-proof}

After recalling notation from \cite{CLoes,CLexp}, we define the evaluation $f(x)$ of a motivic constructible (exponential) function $f$ at a point $x$ and introduce some related notation like  
$f(y,\cdot)$, in line with the notation of Section \ref{sec:def-set}.



\subsection{Notation from \cite{CLoes,CLexp} with small tweaks}\label{sec:recall}

We start by recalling some notation from \cite[Section 2]{CLoes}, \cite[Section 2]{CLexp}. There are some small differences however: we work with  $\Hen$-definable sets instead of some subassignments (see \Cref{rem:T-sub}), and we write $\VF^m\times \RF^n\times \VG^r$, resp.~$\cCexp(X)$ and $\cCe(X)$ instead of $h[m,n,r]$, resp.~$\cC(X)^{\mathrm{exp}}$ and $\cC(X)^{\mathrm{e}}$. 

Let $\cL$ be an expansion by constant symbols of the Denef-Pas language $\LPas$ (in any of the three sorts).
Recall that $\LPas$ is the language with three sorts  $\VF$, $\RF$ and $\VG$ (for valued field, residue field and value group), a relation symbol for the graph of the valuation map from nonzero elements in $\VF$ to $\VG$, an angular component map $\ac$ from $\VF$ to $\RF$, the ring language on $\VF$, 
the ring language on $\RF$, and the language of ordered groups on $\VG$. Recall that an angular component map on a valued field $K$ with residue field $k$ is a multiplicative map $K\to k$ which on units in the valuation ring of $K$ coincides with the projection modulo the maximal ideal.


From now on, and also in Theorems \ref{thm:no:null:intro} and \ref{thm:integrability-intro}, we let $\cS$ be according to the following convention. (Note that in Section \ref{sec:def-set}, $\cS$ was allowed to be more general.)

\begin{convention}\label{defn:S}
We assume that $\cL$ is an expansion by constant symbols of the Denef-Pas language, and that $\Hen$ is a nonempty collection of $\cL$-structures of the form $K=(k\llp t\rrp,k,\ZZ)$, where each such $k$ is a field of characteristic zero and where $K$ carries the natural $\cL$-structure with $\ac(t)=1$ and $\ord(t)=1$. We assume moreover that the collection consisting of the residue fields $k$ of the structures in $\Hen$
forms an elementary class.
\end{convention}      

Recall that a collection is called an elementary class if there is a theory $T$ such that the models of $T$ form precisely the collection.

From here on, $\cS$-definable sets and functions are thus as in \Cref{defn:S-def} for some $\cL$ and $\cS$ as in \Cref{defn:S}.  We consider $\VF$, $\VG$ and $\RF$ as $\Hen$-definable sets.
In detail, $\VF$ is the collection $(\VF_K)_{K \in \Hen}$, where, for $K=(k\llp t \rrp, k , \ZZ)$ in $\cS$,
one has $\VF_K = k\llp t\rrp$, and analogously for $\VG$ and $\RF$. 

\begin{remarksubsub}\label{rem:T-sub}
Let us compare our notion of $\cS$-definable sets with subassignments from \cite{CLoes}. The most flexible variant of subassignments comes from \cite[Section 2.7]{CLoes} and is called `$\cL$-definable $T$-subassignment', with  $k_1$ a choice of base field of characteristic zero, $S_1$ a subring of $k_1\llp t\rrp$ and $T$ a theory in the language of rings with constant symbols for the elements of $k_1$, and, where $\cL$ is the expansion of $\LPas$ with constant symbols for elements of $k_1$ in the residue field sort and for elements in the subring $S_1$ of $k_1\llp t\rrp$ in the valued field sort. This expansion is denoted by $\LPas(S_1)$ in \cite[Section 2.7]{CLoes}. The corresponding $\cS$ would consist of all $\cL$-structures of the form $K=(k\llp t\rrp,k,\ZZ)$, where  $k$ contains $k_1$ and with subring $S_1$ of $k\llp t\rrp$. In \Cref{defn:S} we allow $\cL$ to expand $\LPas$ by constant symbols in a slightly more general way. All results and definitions of \cite{CLoes} and \cite{CLexp} extend naturally to this set-up (indeed, \cite{CLoes} and \cite{CLexp} work throughout with family versions in the sense of relative versions over definable subassignments serving as parameter spaces, which implies variants with general constant symbols by standard model theoretic techniques).   A reader who wants to avoid this extra generality may use the additional assumption throughout the paper that  $\cS$ comes precisely from some choice of  $k_1$, $S_1$ and $T$ as in \cite[Section 2.7]{CLoes} as explained in the above discussion.
\end{remarksubsub}


\begin{remarksubsub}
\label{rem:comparison:fiber}
Definition \ref{defn:points} of points on an $\Hen$-definable set $X$ corresponds to the definition from \cite[Section 2.6]{CLoes} of points on  
subassignments. However, the way we work with a point $x = (x_1, K_1)$ is different: For us, $\cL(x)$ is the expansion by constants for $x_1$ and $\Hen(x)$ consists of structures $\cL(x)$-elementarily equivalent to $K_1$, whereas in \cite[Section 2.6]{CLoes} one uses the expansion
$\cL(K_1)$ of $\cL$ by constants for all elements of $K_1$ and the collection $\Hen(x)' \subset \Hen$ of those structures in $\cS$ that contain $K_1$ (equipped with the natural $\cL(K_1) $-structure). Similarly, in \cite[Section 2.6]{CLoes}, a fiber of an $\Hen$-definable $g\colon Y \to X$ above a point $x$ on $X$ has a different meaning than the $\cS(x)$-definable set that we denote by  $g^{-1}(x)$ (see Section \ref{sec:def-set}).
This also makes a difference for the notion of evaluation of motivic functions at points; see Definition \ref{defn:eval} and Remark~\ref{rem:comparison:eval}. 
\end{remarksubsub}


\subsection{The class $\cC$}\label{sec:C}

Given an $\Hen$-definable set $X$, we will now define the ring $\cC(X)$ of constructible motivic functions on $X$, in line with the rings $\cC(X)$ and $\cC(X,(\cL,T))$ in Sections 5.3 and  16.1 of \cite{CLoes}. The elements of $\cC(X)$ are more informally called functions of $\cC$-class on $X$, and, more formally, they are called motivic constructible functions.

Let $\AA$ be the localisation of the polynomial ring $\ZZ[\LL]$ (in one formal variable $\LL$) by $\LL$ and by all elements of the form $(1-\LL^i)$ for integers $i<0$. Symbolically, one can write $\AA= \ZZ[\LL,\LL^{-1},\bigcup_{i<0} \frac{1}{1-\LL^{i}}]$.

Define $\cP(X)$ as the subring of the ring of all functions sending points $x=(x_1,K_1)$ on $X$ to elements of $\AA$, generated by the following functions
\begin{enumerate}
\item The constant function sending each $x$ to $a$, with  $a$ an element of $\AA$.

\item The function sending $x$ to $\LL^{\alpha(x)}$, where $\alpha:X\to\VG$ is an $\cS$-definable function, and, where $\alpha(x)$ is seen as element of $\ZZ$.

\item  The function sending $x$ to $\alpha(x)$, where $\alpha:X\to\VG$ is an $\cS$-definable function, and, where $\alpha(x)$ is seen as element of $\ZZ$.
\end{enumerate}

Next, define the group $\cQ(X)$ as the quotient of the free abelian group generated by symbols $\lfloor Y\rfloor$ for $\cS$-definable sets $Y\subset X\times \RF^n$ with $n \ge 0$, by the relations
\begin{enumerate}
\item $\lfloor Y\rfloor = \lfloor Y'\rfloor $ if there exists an $\cS$-definable bijection $Y\to Y'$ over $X$ (that is, making a commutative diagram with the coordinate projections from $Y$ and $Y'$ to $X$).

\item $\lfloor Y\cup Y'\rfloor  = \lfloor Y\rfloor  + \lfloor Y'\rfloor $ for disjoint $\cS$-definable subsets $Y$ and $Y'$ of $X\times \RF^n$.
\end{enumerate}

We denote the equivalence class of an element $\lfloor Y\rfloor$ inside $\cQ(X)$ by $[Y]$.
In fact, the group $\cQ(X)$ becomes a ring with product induced by the fiber product over $X$. In detail, one defines $[Y]\cdot[Y']$ by $[Y\times_X Y']$, with $Y\subset X\times \RF^n$ and $Y'\subset X\times \RF^{n'}$ and $Y\times_X  Y'\subset X\times \RF^{n+n'}$ the fiber product of  the projections $Y\to X$ and $Y'\to X$.

Note that the ring $\cQ(X)$ is denoted by $K_0(\RDef_X)$ in \cite[Section 5.1.2]{CLoes}.

Write $\cP^0(X)$ for the subring of $\cP(X)$ generated by the constant function $\LL$, and, by all characteristic functions $\11_Y$ of $\cS$-definable sets $Y\subset X$. There is an injective ring homomorphism $\cP^0(X)\to\cQ(X)$ sending $\LL$ to $[X\times \RF]$ and sending $\11_Y$ to $[Y]$.

 Finally, define $\cC(X)$ as the tensor product of rings
 $$
 \cC(X) := \cQ(X) \otimes_{\cP^0(X)} \cP(X).
 $$

Without harm we write $\LL$ for  $[\RF\times X]$ in $\cC(X)$. Thus, $\LL$ is the class of the affine line in the residue field, over $X$.

\subsection{The classes $\cCe$ and $\cCexp$}\label{sec:Cexp}

Given an $\Hen$-definable set $X$, we will now define the ring $\cCexp(X)$ of constructible exponential motivic functions on $X$, and also the smaller ring $\cCe(X)$, in line with \cite[Sections 3 and 6.2]{CLexp}. Elements of $\cCexp(X)$ are more simply called functions of class $\cCexp$ on $X$; more formally, they are called motivic constructible exponential functions.  We start with defining  $\cQexp(X)$ and $\cQe(X)$.

Define the group $\cQexp(X)$ as the quotient of the free abelian group generated by symbols $\lfloor Y,\xi,h \rfloor$ for $\cS$-definable sets $Y\subset X\times \RF^n$ with $n\ge 0$ and $\cS$-definable functions $\xi:Y\to \RF$, $h:Y\to\VF$, by the relations
\begin{enumerate}
\item $\lfloor Y,\xi,h\rfloor  = \lfloor Y',\xi',h'\rfloor $ if there exists an $\cS$-definable bijection $f:Y\to Y'$ over $X$  such that $\xi'\circ f= \xi$ and $h'\circ f = h$.

\item $\lfloor Y\cup Y', \xi\cup \xi',h\cup h'\rfloor   = \lfloor Y,\xi,h\rfloor  + \lfloor Y',\xi',h'\rfloor $ for disjoint $\cS$-definable subsets $Y$ and $Y'$ of $X\times \RF^n$, and where $\xi\cup \xi':Y\cup Y'$ sends $y$ to $\xi(y)$ if $y\in Y$ and to $\xi'(y)$ if $y\in Y'$, and similarly for $h\cup h'$.

\item $\lfloor Y,\xi,h_1+h_2\rfloor  = \lfloor Y,\xi+\bar h_1,h_2\rfloor $ where the $h_i:Y\to\VF$ are $\cS$-definable functions for $i=1,2$ such that furthermore $\ord h_1(y) \ge 0$ for all $y$ in  $Y$, and where $\bar h_1$ the reduction of $h_1$ modulo the maximal ideal.

\item $\lfloor Y\times \RF,\xi+p,h\rfloor   =0$ when $h : Y\times \RF \to \VF$ and $\xi: Y\times \RF \to \RF$ both factorize through the projection $Y\times \RF \to Y$ and where $p:Y\times \RF\to\RF$ is the projection.
\end{enumerate}

Again, we denote the equivalence class of an element $\lfloor Y,\xi,h\rfloor$ inside $\cQexp(X)$ by $[Y,\xi,h]$.

The group $\cQexp(X)$ becomes a ring with product induced by the fiber product over $X$. In detail, one defines $[Y,\xi,h]\cdot[Y',\xi',h']$ by $[Y\otimes_X Y',\xi\circ p_Y +  \xi'\circ p_{Y'},h\circ p_Y + h'\circ p_{Y'}]$, with $Y\otimes_X  Y'\subset X\times \RF^{n+n'}$ the fiber product of  the projections $Y\to X$ and $Y'\to X$ and with $p_Y:Y\otimes_X Y'\to Y$ and $p_{Y'}:Y\otimes_X Y'\to Y'$ the projections.

We define $\cQe(X)$ as the subring of $\cQexp(X)$  generated by the elements $[Y,\xi,0]$. 
Note that the ring $\cQexp(X)$ is denoted by $K_0(\RDef_X^{\rm exp})$ in \cite[Section 3]{CLexp}, and $\cQe(X)$ by $K_0(\RDef_X^{\rm e})$ in \cite[Section 6.2]{CLexp}.

Finally define the rings
$$
 \cCexp(X) := \cQexp(X) \otimes_{\cQ(X)} \cC(X)
 $$
 and
$$
 \cCe(X) := \cQe(X) \otimes_{\cQ(X)} \cC(X).
 $$
There are natural inclusions of rings
\begin{equation}\label{eq:incl}
\cC(X)\to\cCe(X)\to\cCexp(X),
\end{equation}
by \cite[Lemma 3.3.1]{CLexp}.

\subsection{Integrability}\label{sec:integrab}

We rely on \cite{CLoes,CLexp} to fix our notion of integrability, without recalling the full details. In fact, after Theorem 2 and our criteria from Propositions \ref{prop:crit} and \ref{lem:exp-to-con}, the original definition of integrability for functions of the classes $\cC$, $\cCe$ and $\cCexp$ in \cite{CLoes,CLexp}  seem to become less important.


\begin{defnsubsub}[Dimension]\label{defndim}
Let $d\ge 0$ be an integer. For an $\Hen$-definable set $X$ with $X\subset \VF^m\times \RF^n\times \VG^r$, say that $X$ has dimension at most $d$ if for each $K$ in $\cS$, there exists a $\VF_K$-linear function $f:\VF_K^m \to \VF_K^d$ such that $f$ has finite fibers on $A_K$, with $A_K$ the image of $X_K$ under the coordinate projection $X_K\to \VF_K^m$. Say that $X$ has dimension $d$ if moreover $f(A_K)$ has nonempty interior in $\VF_K^d$ for some $K$ in $\cS$ and some choice of linear function $f$, with the valuation topology on $\VF_K^d$. If an  $\cS$-definable function $g:X\to Z$ is given, say that $X$ has relative dimension at most $d$ (along $g$) if the fibers of $g$ have dimension at most $d$.
\end{defnsubsub}

\begin{defnsubsub}[Integrability]\label{def:integrab}
Let $f$ be in $\cC(X)$, in $\cCe(X)$ or in $\cCexp(X)$ and let $g:X\to Z$ be an $\Hen$-definable function whose fibers have dimension at most $d$. One calls $f$ integrable in relative dimension $d$ over $Z$ along $g$ (namely in the fibers of $g$) if the class of $f$ in $C^{d}(X\to Z)$ lies in $\mathrm{I}_ZC(X\to Z)$, resp.~the class of $f$ in $C^{d}(X\to Z)^{\rm exp}$ lies in $\mathrm{I}_ZC(X\to Z)^{\rm exp}$ with notation from \cite[Section 14]{CLoes}, resp.~\cite[Section 4.3]{CLexp}.
One simply says integrable in dimension $d$ (instead of integrable in relative dimension $d$ over $Z$ along $g$) if $Z_K$ is a singleton for each $K$ in $\Hen$.
\end{defnsubsub}

We will give self-contained criteria for the notion of relative integrability in Propositions \ref{prop:crit} and \ref{lem:exp-to-con}.
The more general notion of integrability of $f$ in relative dimension $d'$ over $Z$ for a given $d'$ with $0\le d'\le d$ requires in particular that the support of $f$ is contained in an $\cS$-definable $X'$ set of relative dimension at most $d'$ over $Z$. But in such a case one can usually simply restrict $f$ to $X'$ and work as in \Cref{def:integrab} again, with $d=d'$. This more general notion of integrability (namely in smaller relative dimensions) is not used outside of Proposition \ref{prop:VF-RF}, where its meaning is recalled. 

\subsection{Evaluation in points}\label{sec:eval}


Let  $\cL$ and $\Hen$ be as in Section \ref{sec:recall}, let $Z$ be an $\Hen$-definable set, and let $z$ be a point on $Z$. We use notation from Sections \ref{sec:def-set} and \ref{sec:recall}.
By the definition of constructible motivic functions, any $f$ in $\cC(X)$ is built up from some $\Hen$-definable sets and functions, and thus determines a constructible motivic function $f_{\Hen(z)}$ in $\cC(X_{\Hen(z)})$, where $\cS(z)$ is as in Section \ref{sec:def-set}. Indeed, this is well-defined since going to $\Hen(z)$, all existing Grothendieck ring relations are preserved, and it clearly yields a ring homomorphism from $\cC(X)$ to $\cC(X_{\Hen(z)})$.
Similarly, $f$ in $\cCe(X)$ or in $\cCexp(X)$ determines $f_{\Hen(z)}$ in $\cCe(X_{\Hen(z)})$ resp.~in $\cCexp(X_{\Hen(z)})$.

Recall that given an $\Hen$-definable set $X$ and an $\Hen$-definable subset $Y \subset X$, we have a notion
of restriction of a constructible function $f$ in $\cC(X)$ to $f|_Y$ in $\cC(Y)$ (and analogously for $\cCe(X)$ and $\cCexp(X)$). More formally, $f|_Y$ equals the pull-back $g^*(f)$ of $f$ along the
inclusion map $g: Y \to X$. We define the evaluation of a motivic function at a point as the restriction of the function to the corresponding one-point set:

\begin{defnsubsub}[Evaluation]\label{defn:eval}
For  $f$ in $\cC(X)$ and $x$ a point on $X$, we define $f(x)$ in $\cC(\{x\})$ as the restriction
of $f_{\Hen(x)}$ to $\{x\}$.
(Note that $\{x\}$ is an $\Hen(x)$-definable subset of $X_{\Hen(x)}$.)
One defines $f(x)$ likewise when $f$ lies in $\cCe(X)$ or in $\cCexp(X)$.
\end{defnsubsub}

Note that evaluation at a given point is a ring homomorphism (since the pull-back is).


\begin{remarksubsub}[Comparison of evaluation with notions from \cite{CLoes,CLexp}]\label{rem:comparison:eval}
Given a point $x$ on an $\cS$-definable set $X$ and a function $f$ in $\cC(X)$, $\cCe(X)$, or $\cCexp(X)$,
there is also a notion of evaluation of $f$ at $x$
in \cite{CLoes,CLexp}, denoted by $i^*_x(f)$ in Section 5.4 of \cite{CLoes} and in (4.3.1) of \cite{CLexp}. This notion is defined differently than in our Definition~\ref{defn:eval}, instead using $\Hen(x)'$ from
Remark \ref{rem:comparison:fiber}; this makes a difference for the ring in which the value of $f$ at $x$ lies. In detail, if we write $\{x\}'$ for the $\Hen(x)' $-definable set given by a point $x$ on $X$, then $i^*_x(f)$ from \cite{CLoes,CLexp}  lies in $\cC(\{x\}')$, where $f(x)$ from \Cref{defn:eval} lies in $\cC(\{x\})$, and similarly for $\cCe$ and  $\cCexp$.
%
\end{remarksubsub}

For an $\Hen$-definable function $g:X\to Z$ we denote by $f_{|g^{-1}(z)}$ the restriction of $f_{\Hen(z)}$ to $g^{-1}(z)$.
When $p:X\subset Y\times Z \to Z$ is the projection to the $Z$-coordinate with $X,Y,Z$ some $\Hen$-definable sets, $z$ a point on $Z$, and
$f$ in $\cC(X)$, then we write $f(\cdot,z)$ for the restriction of  $f_{\Hen(z)}$ to $X_z$ with $X_z$ being $p^{-1}(z)$ as at the end of \Cref{sec:def-set}; note that $f(\cdot,z)$ lies  in $\cC(X_z)$.
We use similar notation for $f$ in $\cCe(Z)$ and in $\cCexp(Z)$.

\begin{example}\label{ex:null:i^*}
Here is an example where Theorem \ref{thm:no:null:intro} would fail for the notion of $i^*_x(f)$ from Section 5.4 of \cite{CLoes} instead of our notion of $f(x)$ (see Remark \ref{rem:comparison:eval} for notation). Let $\cS$ be a collection of structures as before, thus in particular of the form $K=(k\llp t\rrp,k,\ZZ)$, where we require moreover that all occurring $k$ are algebraically closed and where we let $\cL$ be $\LPas$ (without extra constant symbols). Let $X$ be $\RF^\times$, that is, the $\cS$-definable set which is the multiplicative group of the residue field. Consider the definable set $Z\subset \RF\times X$, given by the formula $z^2=x$, for variables $(z,x)$ running over $\RF\times X$, and, consider the class $[Z]$ of $Z$ in $\cC(X)$. Let $f$ be $[Z]-2$ in $\cC(X)$. Then, with notation from Remark \ref{rem:comparison:eval}, for each point $x$ on $X$ we have $i^*_x(f)=0$. Indeed, in $i^*_x(f)$, one has constant symbols in particular for each square root of $x$, so that there is a definable bijection from the set of square roots of $x$ to the two point set $\{0,1\}$, for nonzero $x$ in an algebraically closed field. However, $f$ is nonzero in $\cC(X)$. Indeed, 
one has $f(-1)\not=0$ by Proposition 5 of \cite[Section 5]{LiuSebag}, with our evaluation of $f$ at $-1$ from \Cref{defn:eval}. 
\end{example}

\subsection{Some consequences of the main results}

Recall from \cite[Section 5.4]{CLoes}, \cite[Section 3.4]{CLexp} (or from the overview \cite[Section 2]{Raib-motivic}) that, for any $\Hen$-definable function $g:X\to Z$, there are natural pull-back maps
$g^*:\cC(Z)\to\cC(X)$, $g^*:\cCe(Z)\to\cCe(X)$,  and $g^*: \cCexp(Z)\to\cCexp(X)$, which are defined in \cite[Section 5.4]{CLoes}, \cite[Section 3.4]{CLexp} as an abstract form of composition. By Theorem \ref{thm:no:null:intro}, this can  now be seen as a concrete form of composition, as follows.

\begin{cor}\label{cor:eval:hom:pull}
Let $X$, $Z$ be $\Hen$-definable sets, let $g:X\to Z$ be an $\Hen$-definable function and let $f$ be in $\cC(Z)$.
Then the pull-back $g^*(f)$ (as defined in \cite[Section 5.4]{CLoes}) is the unique function in $\cC(X)$ satisfying
\begin{equation}\label{eq:pullback}
(g^*(f))(x) = f(g(x))_{\cS(x)}
\end{equation}
for each point $x$ on $X$, where $f(g(x))_{\cS(x)}$ is  the base change of $f(g(x))$ to ${\cS(x)}$ as defined just below \Cref{defn:x}.
The same statement holds with $\cCe$ or with $\cCexp$ instead of $\cC$, with pull-back as defined in \cite[Section 3.4]{CLexp}.
\end{cor}
Note that on the right hand side of \eqref{eq:pullback}, formally, $f(g(x))$ is an element of $\cC(\{g(x)\})$, where $\{g(x)\}$ is considered as an $\cS(g(x))$-definable set, but we can canonically identify $\cS(g(x))(x)$ with $\cS(x)$.

\begin{proof}[Proof of Corollary~\ref{cor:eval:hom:pull}]
The pull-back satisfies (\ref{eq:pullback}),
since the pull-back of a composition is equal to the composition of pull-backs (note that restriction to $\{x\}$ is a pull-back along the inclusion map $\{x\}\to X_{\cS(x)}$). Theorem~\ref{thm:no:null:intro} implies that this already determines $g^*(f)$.
\end{proof}

\bigskip

We fix now what we mean by relative integrals, according to \cite[Sections 14.1 -- 14.3]{CLoes} and \cite[Section 4.3]{CLexp}.  We again restrict our attention to fixed relative dimension, which is the key case for reasons already explained  after \Cref{def:integrab}.
\begin{defnsubsub}[Relative motivic integrals]\label{def:integral}
Consider an $\Hen$-definable function $g:X\to Z$ with fibers of dimension at most $d$ and a function $f$ in $\cC(X)$ which is integrable in relative dimension $d$ over $Z$ along $g$ according to \Cref{def:integrab}. Let $ [f]$ be the class of $f$ in $C^{d}(X\to Z)$, with notation from  \cite[Sections 14.2]{CLoes}. Consider  $g_{!Z}([f])$ in $\cC(Z)$  from  \cite[Sections 14.2]{CLoes} and denote it by  $\mu_{Z,d}(f)$. We call $\mu_{Z,d}(f)$ the relative integral of $f$ in relative dimension $d$ over $Z$ along $g$.
We define $\mu_{Z,d}(f)$ in $\cCe(Z)$, resp.~$\cCexp(Z)$, similarly for $f$ in $\cCe(X)$, resp.~$\cCexp(X)$, using \cite[Section 4.3]{CLexp}.
\end{defnsubsub}
The relative integral $\mu_{Z,d}(f)$ as in \Cref{def:integral} is a form of integration over each fiber of $g$. Note that, in \cite[Sections 14.2]{CLoes}, $g_{!Z}([f])$ is also denoted by $\mu_{Z}([f])$, but we denote it by $\mu_{Z,d}(f)$. 

\begin{cor}\label{cor:eval:push}
Let $X$ and $Z$ be $\Hen$-definable sets, let $f$ be in $\cC(X)$, and let $g:X\to Z$ be an $\Hen$-definable function with fibers of dimension $\leq d$ for some integer $d\geq 0$.
Suppose that for each point $z$ on $Z$, the restriction $f_{|g^{-1}(z)}$ is integrable in dimension $d$, as in \Cref{def:integrab}. Then $f$ is integrable in relative dimension $d$ along $g$, and, with notation from \Cref{def:integral}, $\mu_{Z,d}(f)$ it is the unique function in $\cC(Z)$ satisfying
\begin{equation}\label{eq:pushforward}
\mu_{Z,d}(f)(z) = \mu_{z,d}(f|_{g^{-1}(z)})
\end{equation}
for each point $z$ on $Z$.
The same statement holds when $f$ lies in $\cCe(X)$ or in $\cCexp(X)$.
\end{cor}

\begin{proof}
The relative integrability in relative dimension $d$ follows from Theorem~\ref{thm:integrability-intro}. That it satisfies (\ref{eq:pushforward}) follows from 
\cite[Equation (4.3.1)]{CLexp} applied to $f_{\cS(z)}$, and the inclusions from (\ref{eq:incl}). Indeed, $\mu_{Z,d}$ from \Cref{def:integral} maps $\cC(X)$ to $\cC(Z)$ and $\cCe(X)$ to $\cCe(Z)$ by (A1) of \cite[Theorems 4.1.1, 4.3.1, Proposition 6.2.1]{CLexp}.
The uniqueness comes from Theorem~\ref{thm:no:null:intro}.
\end{proof}
%
%
\begin{remarksubsub}\label{rem:14.2.2}
Note that Corollary 14.2.2 of \cite{CLoes} looks similar to a part of Corollary \ref{cor:eval:push}, but then for  $\cC_+$ (the non-negative motivic constructible functions) instead of for $\cC$, $\cCe$ and $\cCexp$ and with $i_x^*$ instead of our notion of evaluation at $x$ (see Remark \ref{rem:comparison:eval}).
It seems unlikely that this result of 14.2.2  from \cite{CLoes} could imply directly our results for $\cC$ and $\cCe$. Note that, in any case, our results for $\cCexp$ are completely different from Corollary 14.2.2 of \cite{CLoes}, and, require very different work.
\end{remarksubsub}

\bigskip

Theorem \ref{thm:no:null:intro} implies a slightly more general variant, as follows.
\begin{cor}\label{cor:no:null:main}
Let $g:X\to Z$ be an $\Hen$-definable function for some $\Hen$-definable sets $X$ and $Z$ and let $f$ be in $\cC(X)$.
Then the following statements are equivalent.
\begin{enumerate}
\item\label{item:1:0} One has  $f_{|g^{-1}(z)}=0$ for all points $z$ on $Z$.
\item\label{item:2:0} One has $f=0$.
\end{enumerate}
 Furthermore, the same equivalence holds for $f$ in  $\cCe(X)$ and in $\cCexp(X)$.
\end{cor}
\begin{proof}
The implication ``(2) $\Rightarrow$ (1)'' is trivial. For the other one, note that given any point $x$ on $X$,
we obtain a point $z = g(x)$ on $Z$, and for this $z$,
$f_{|g^{-1}(z)} = 0$ implies $f(x) = 0$.
Thus we have $f(x) = 0$ for all points $x$ on $X$, so $f = 0$ follows from Theorem \ref{thm:no:null:intro}.
\end{proof}

\bigskip

Theorems~\ref{thm:no:null:intro} and \ref{thm:integrability-intro} together also imply the projection formula of Theorems 1.1 and 2.28 from \cite{CelyRaib}; for simplicity, we only cite the ``non-relative version'' from the introduction of \cite{CelyRaib}, and, in fixed dimension $d$ as in the above definitions \ref{def:integrab}, \ref{def:integral}.

\begin{cor}[Projection formula {\cite[Theorem 1.1]{CelyRaib}}]\label{cor:CelyRaib}
Let $X$, $W_1$, $W_2$ and $\gamma\colon W_1 \to W_2$ be $\Hen$-definable and let $f$ be in $\cC(W_2 \times X)$. Let $X$ be of (valued field) dimension $d$. We write
$\pi_i\colon W_i \times X \to W_i$ for the projection and set $\tilde\gamma := (\gamma\times\operatorname{id}_X)\colon W_1 \times X \to W_2 \times X$ and we use  Definitions \ref{def:integrab}, \ref{def:integral}.
\begin{enumerate}
 \item If $f$ is integrable in relative dimension $d$ over $W_2$ along $\pi_2$ (namely in the fibers of $\pi_2$), then $\tilde\gamma^*(f)$ is integrable in relative dimension $d$ over $W_1$ along $\pi_1$; if $\gamma$ is surjective, then the converse also holds.
 \item If (1) holds, then we have $\gamma^* (\mu_{W_2,d}(f))
= \mu_{W_1,d}(\tilde\gamma^*(f))$.
\end{enumerate}
The same holds for $f$ in $\cCe(W_2 \times X)$ or $\cCexp(W_2 \times X)$ instead of $\cC$.
\end{cor}

\begin{proof}
(1) follows from Theorem~\ref{thm:integrability-intro}; indeed, it allows one to check relative integrability in the fibers of $\pi_i$ on the points of $W_i$.
For (2), note that Theorem~\ref{thm:no:null:intro} allows us to reduce to the case where $W_1$ and $W_2$ are both singletons (we need to prove the equality at each point $w_1$ on $W_1$, so we can replace $W_1$ by $\{w_1\}$ and $W_2$ by $\{\gamma(w_1)\}$), and this case is immediate.
\end{proof}

\section{Integrability revisited}\label{sec:int:revisited}

The following propositions give a direct viewpoint on integrability, in comparison with the (less direct) existence/uniqueness results of integrable motivic constructible (exponential) functions from \cite[Thm. 10.1.1, Prop. 13.2.1, Thm. 14.1.1]{CLoes}, \cite[Theorems 4.1.1, 4.3.1]{CLexp}. Together with
our main theorems, 
they lead to a simplified understanding of the framework of motivic integration of \cite{CLoes,CLexp}.

The proposition  just below is based on the Rectilinearization of Presburger sets in families from \cite{CPres}. It will allow us to formulate a first criterion for integrability in Proposition \ref{prop:crit}. This criterion for integrability is then complemented, for general functions of $\cCexp$-class, by Proposition \ref{lem:exp-to-con}. Furthermore, integrability over valued field variables is reduced to integration over the residue field and the value group, by Propositions \ref{prop:twisted:boxes} and \ref{prop:VF-RF}.

Write $\NN$ for the set of natural numbers (the nonnegative integers). We sometimes consider $\NN$ as the $\Hen$-definable subset $\VG_{\ge 0}$ of $\VG$.



%

\begin{prop}
\label{lem:rec}
Let $X$ be an $\cS$-definable subset of $\VG^{R} \times Z$ for some $R\ge 0$ and some $\cS$-definable set $Z$. Let $f_j$ be in $\cC(X)$ or in $\cCe(X)$, for $j$ in a finite set $J$.
Then, there exist a finite partition of $X$ into $\Hen$-definable sets $X_i$, 
$\Hen$-definable bijections
$$
\theta_i:  \VG_{\ge 0}^{r_i}\times A_i \to X_i
$$
for some $r_i\le R$ and some $\Hen$-definable sets $A_i\subset \VG^{R-r_i}\times Z$ with the following properties for each $i$:
\begin{itemize}
 \item For every point $z$ on $Z$, the fiber
 $A_{i,z}$ is finite.
 \item There exists a matrix $M_i \in \ZZ^{R \times r_i}$ and an $\Hen$-definable function $d_i\colon A_i \to \VG^R$ such that $\theta_i$ is of the form
 \begin{equation}\label{theta:linear}
   (x,y, z) \mapsto (M_ix + d_i(y,z), z),
 \end{equation}
  for $x \in \VG_{\ge 0}^{r_i}$, $y\in\VF^{R-r_i}$ and $z \in Z$ satisfying $(y,z) \in A_i$.
\item
Write  $g$ for the coordinate projection  $\VG_{\ge 0}^{r_i}\times A_i \to \VG^{r_i}$ onto the first $r_i$ coordinates. For each $j \in J$,
there exist a finite set
$$
L_{ij}\subset \NN^{r_i} \times \ZZ^{r_i},
$$
and nonzero $c_{a,b,i,j}$ in $\cC(A_i)$  (resp.~in $\cCe(A_i)$) such that
$\theta_i^*(f_j)$
is of the form
\begin{equation}\label{f:sum:int:prop}
\sum_{(a,b) \in L_{ij}}  c_{a,b,i,j}\cdot  g^a \cdot  \LL^{b\cdot  g}.
\end{equation}
Here, $g^a$ stands for $g_1^{a_1}\cdots g_{r_i}^{a_{r_i}}$ and
$b \cdot g$ stands for $b_1\cdot g_1 + \dots+ b_{r_i} \cdot g_{r_i}$.
\end{itemize}
Furthermore, a sum of the form \eqref{f:sum:int:prop} (with all $c_{a,b,i,j} \ne 0$) is equal to $0$ if and only if $L_{ij} = \emptyset$.
\end{prop}
\begin{proof}
The first part (about existence) follows from the parametric rectilinearization result \cite[Theorem 3]{CPres} and the quantifier elimination result from \cite{Pas}, namely as follows: Using quantifier elimination, one obtains that each $f_j$ can be written as a finite sum of expressions of the form
$c \cdot  \prod_{\ell=1}^s h_\ell \cdot  \LL^{h_0}$,
for $c$ in $\cC(Z)$ or $\cCe(Z)$ and where $h_0,\dots, h_s$ are Presburger functions depending only on the first $R$ ($\VG$-)coordinates. Now apply paramatric rectilinearization to the graph of the function sending $x \in \VG^R$ to the tuple of all $h_\ell(x)$, for all $h_\ell$ appearing in all the $f_j$.

The furthermore part follows
by the  definitions of $\cC$ and  $\cCe$ as tensor products in Section \ref{sec:C} and \ref{sec:Cexp}. Indeed,
the family of functions $(g^a\cdot \LL^{b\cdot g})_{(a,b) \in \NN^{r_i} \times\ZZ^{r_i}}$ in $\cP(\VG^{r_i}_{\ge 0}\times A_i)$ is linearly independent over $\cP^0(\VG^{r_i}_{\ge 0}\times A_i)$, so no non-trivial sum of the form
\eqref{f:sum:int:prop} becomes zero in those tensor products.
\end{proof}

We introduce some shorthand notation, similar to \cite[Section 3.1.2]{CLexp}.
For $h:X\to \VF$ an $\cS$-definable function on an $\cS$-definable set $X$, we write $E(h)$ as shorthand in $\cCexp(X)$ for the ``motivic additive character'' evaluated in $h$, which is in full denoted by $[{\rm id}: X\to X,0,h]$ in \ref{sec:Cexp}, with ${\rm id}$ the identity map.

\begin{prop}\label{prop:crit}
Let $X$, $f=f_1$, $R$ and $Z$ be as in Proposition \ref{lem:rec}, with furthermore $Z = \RF^n \times Z'$ for some $\Hen$-definable $Z'$ and $n \ge 0$. Let $h:X\to\VF$ be an $\cS$-definable function, and let $F$  be the function
$F =f\cdot E(h)$ in  $\cCexp(X)$. Then one has that $f=0$ if and only if $F=0$.
Furthermore, the following conditions are equivalent:

\begin{itemize}

\item[(i)] $f$ is integrable in relative dimension $0$ over $Z'$, along the projection $X\to Z'$.
\\

\item[(ii)]  There exist a finite partition of $X$ into pieces $X_i$ as in Proposition \ref{lem:rec} and with its further objects like $\theta_i$, $r_i$ and $L_i=L_{i1}$ and its conditions for $f_1=f$ and $J=\{1\}$, such that moreover
\begin{equation}\label{cond:L}
L_i\subset \NN^{r_i} \times (\ZZ \setminus \NN)^{r_i} \mbox{ for each } i.
\end{equation}
\\

\item[(iii)]  For each way of writing $X$ as finite partition into pieces $X_i$ as in Proposition \ref{lem:rec} and with its further notation and conditions for $f_1=f$ and $J=\{1\}$,
the inclusions as in  (\ref{cond:L}) hold. 
\\

\item[(iv)] $F$ is integrable in relative dimension $0$ along the projection $X\to Z'$.
\\
\end{itemize}
\end{prop}
The inclusion from (\ref{cond:L}) corresponds to summability when one replaces $\LL$ by real $q>1$, where, obviously, $\alpha\mapsto \alpha^aq^{b\alpha}$ is summable over $\alpha\in \NN$ for some $(a,b)\in \NN \times \ZZ$ if and only if $b\in\ZZ\setminus\NN$. Note that taking $0$ as relative dimension for the projection map $X\to Z'$ from \Cref{prop:crit} is natural; indeed, dimensions are taken in the valued field.
\begin{proof}[Proof of Proposition \ref{prop:crit}]
The first statement about the equivalence of $F=0$ and $f=0$ is clear by definition of $\cCexp(X)$ from \cite{CLexp}. (Note that $f=F\cdot E(-h)$, using the inclusions $\cC(X)\subset \cCe(X)\subset \cCexp(X)$.)
The equivalence between (i) and (iv) is similarly clear, as well as the implications from (iii) to (ii) to (i) by the definitions of relative integrability from \cite[Theorems   10.1.1, 14.1.1]{CLoes}, \cite[Theorems 4.1.1, 4.3.1]{CLexp}. (Indeed, replacing $\LL$ by any real $q>1$ yields clearly summable functions when $L_{ij}$ satisfies (\ref{cond:L}).)
Let us thus suppose that (i) holds for $f$ and prove (iii) for $f$,
so let $X_i$ and $\theta_i$ be given.
By working piecewise and pulling back along the maps $\theta_i$, we may suppose that $X$ and $f$ itself are already of the following form:
$$
X = \VG_{\ge0}^{r}\times A
$$
for some $r\le R$ and some $\Hen$-definable set $A\subset \VG^{R-r}\times \RF^n\times Z'$
such that $A_{w}$ is a finite set for every point $w$ on $\RF^n\times Z'$, and, such that
$f$
is of the form
\begin{equation}\label{f:sum:int:prop:again}
f=\sum_{(a,b) \in L}  c_{a,b}\cdot  g^a \cdot  \LL^{b\cdot  g}
\end{equation}
with $L\subset \NN^{r} \times \ZZ^{r}$ a finite set,
nonzero $c_{a,b}$ in $\cC(A)$  (resp.~in
$\cCe(A)$), and where $g$ is the
coordinate projection  $X\to \VG^{r}$.
Our goal is to prove that $L$ is a subset of $\NN^r\times(\ZZ\setminus\NN)^r$.
We will show how to prove $L \subset \NN^r\times (\ZZ\setminus\NN) \times \ZZ^{r-1}$, the other coordinates working similarly.

Relative integrability of $f$ over $Z'$ in relative dimension $0$ in particular implies relative integrability over $Z := \VG_{\ge 0}^{r-1} \times A$, with respect to the projection $\pi_{\ge2}(x_1, \dots, x_r, a) := (x_2, \dots, x_r, a)$.
The definition of relative
integrability in relative dimension $0$ implies that $f$ is equal to a finite sum of
products $d_j \cdot h_j$, for some nonzero $d_j \in \cC(Z)$ (resp.~in $\cCe(Z)$)  and some $h_j \in \mathrm{I}_{Z}\cP(X)$, where $\mathrm{I}_{Z}\cP(X)$ stands for the $\RF^n\times Z$-integrable constructible Presburger functions on $X$ (namely, summable in the fibers of $\pi_{\ge2}$), as defined in \cite[Section~4.5]{CLoes}.

Apply Proposition \ref{lem:rec} to the functions $h_j$ over this set $Z$ (i.e., taking $R = 1$ and considering $X$ as a subset of $\VG^1 \times Z$).
This in particular yields maps  $\theta_i:   \VG_{\ge 0}^{r_i}\times A_i \to X_i
$ (for some $r_i \in \{0,1\}, A_i,X_i$). Since the union of all the $X_i$ is equal to $X$, we have $r_i = 1$ for at least one $i$. Fix such an $i$ for the remainder of the proof. In particular, for that particular $i$, we have $A_i\subset Z$, and $\theta_i$ is of the form
\[
\theta_i\colon
\VG_{\ge0} \times A_i \to \VG_{\ge0} \times Z,
(x, z) \mapsto (e\cdot x + d(z), z)
\]
for some $e \in \ZZ$ and some definable $d\colon A_i \to \VG$. Moreover, one easily sees that $e \ge 1$.

Pulling back \eqref{f:sum:int:prop:again} yields
\begin{equation*}
\theta_i^*(f) = \sum_{(a,b) \in L}  c_{a,b}\cdot\tilde g^a \LL^{b\cdot \tilde g}
\end{equation*}
for $\tilde g := g \circ \theta_i$, which is the map sending $(x_1, \hat x, w) \in \VG \times \VG^{r-1} \times A$ to $(ex_1 + d(\hat x, w), \hat x)$. Regrouping that sum according to the occurence of $x_1$ yields a sum of the form
\begin{equation}\label{gtilde}
\sum_{(a',b') \in L'}  c'_{a',b'}\cdot g_1^{a'} \LL^{b'\cdot g_1}
\end{equation}
for some $L' \subset \NN \times \ZZ$, where $g_1\colon  \VG_{\ge 0}^{r_i}\times A_i \to \VG$ is the projection to the first coordinate, and where $c'_{a',b'}$ live on $A_i$. Moreover,
to obtain our desired conclusion that $L$ is a subset of $\NN^r\times (\ZZ\setminus\NN) \times \ZZ^{r-1}$, it now suffices to prove that $L'$ is a subset of $\NN \times (\ZZ\setminus\NN)$ (since $e \ge 1$).

Our application of Proposition \ref{lem:rec} to the $h_j$ also yields sets $L_{ij} \subset \NN\times \ZZ$ over which the sums \eqref{f:sum:int:prop} run (still for the specificly chosen $i$ satisfying $r_i = 1$).
Since $h_j$ is summable over $Z$, so is
$\theta_i^*(h_j)$, so we have
$L_{ij} \subset \NN\times(\ZZ\setminus\NN)$ for all $j$.
This yields a similar statement for the pullback of $f$:
\[
\theta_i^*(f) = \sum_j \theta_i^*(d_j)\theta_i^*(h_j)
=\sum_j d_j\sum_{(a',b') \in L_{ij}} c_{a',b',i,j}\cdot  g_1^{a'} \cdot  \LL^{b'\cdot  g_1}
,
\]
which we can rewrite as
\[
\dots =\sum_{(a',b') \in L_{i}} c''_{a',b'}\cdot  g_1^{a'} \cdot  \LL^{b'\cdot  g_1}
\]
for some $c''_{a',b'}$ in $\cC(A_i)$ (resp.~in $\cCe(A_i)$), and for $L_i \subset \bigcup_j L_{ij} \subset  \NN\times(\ZZ\setminus\NN)$.
Comparing the coefficients of this with those of \eqref{gtilde} (which we can, by the furthermore part of Proposition~\ref{lem:rec}), we obtain $L' \subset \NN \times (\ZZ\setminus\NN)$, as desired.
\end{proof}

In the following proposition, $E(h)$ is shorthand notation as recalled just above \Cref{prop:crit}.

\begin{prop}[Reduction from $\cCexp$ to  $\cCe$]\label{lem:exp-to-con}
Let $f$ be in $\cCexp(X)$ for some $\Hen$-definable set $X$. Then there exist an $\Hen$-definable set $\widetilde X\subset X\times \RF^r$ for some $r\geq 0$
and $\tilde f$ in $\cCexp(\widetilde X)$ such that the projection
$p:\widetilde X \to X$ is surjective with finite fibers and such that 
$$
\tilde f = E(h)c
$$
for some $\Hen$-definable function $h:\widetilde X \to \VF$ and some $c$ in $\cCe(\widetilde X)$, 
and such that furthermore
$$
\mu_{X,0}(\tilde f)=f
$$
and, for all points $a,b$ on $\widetilde X$ with $p(a)=p(b)$,
\begin{equation}\label{eq:rel}
\ord (h(a) - h(b) ) <0,
\end{equation}
where $\mu_{X,0}(\tilde f)$ is the relative integral of $\tilde f$ in relative dimension $0$ along $p$ from \Cref{def:integral}.    

Moreover, for any such choice of $\widetilde X$  and $\tilde f$,
for each point $x$ on $X$, each $d\ge 0$, and each  $\cS$-definable function  $g:X\to Z$ whose fibers have dimension at most $d$, the following statements hold.
\begin{itemize}
\item[(a)]\label{item:1} One has 
$$
f=0 \mbox{ if and only if } \tilde f =0.
$$ 

\item[(b)]\label{item:2} One has 
$$
f(x)=0 \mbox{ if and only if } \tilde f|_{p^{-1}(x)}=0.
$$ 

\item[(c)]\label{item:3} 
The function $f$ is integrable in relative dimension $d$ over $Z$ along $g$ if and only if $\tilde f$ is integrable in relative dimension $d$ over $Z$ along $g\circ p$.
 \end{itemize}
\end{prop}

Note that $\mu_{X,0}(\tilde f)$ from \Cref{lem:exp-to-con} coincides with the push-forward $p_!(\tilde f)$ of $\tilde f$ along $p$ as in Section 3.6 of \cite{CLexp}.

\begin{proof}[Proof of Proposition~\ref{lem:exp-to-con}]
By the definition of $\cCexp(X)$ in \Cref{sec:Cexp}, 
$f$ is a finite sum of terms of the form
$$
H_i\otimes [p_i: Y_i\to X,\xi_i,h_i]
$$
with $H_i$ in $\cC(X)$, $Y_i\subset X\times \RF^{r_i}$ an  $\Hen$-definable set for some $r_i$, and $\Hen$-definable functions $\xi_i:Y_i\to \RF$,  $h_i:Y_i\to \VF$, 
and projection $p_i:Y_i\to X$.
Consider a point $x$ on $X$ and let $G_x:= \bigcup_i h_i(p_i^{-1}(x))$ be the union over $i$ of the images of the restrictions of $h_i$ to $p_i^{-1}(x)$ (strictly speaking, the restriction of $h_{i,\Hen(x)}$ to $p_i^{-1}(x)$).
Then $G_x$ is a finite $\Hen(x)$-definable set (indeed, definable functions from the residue field into the valued field have finite image, by quantifier elimination \cite{Pas}).
We will now construct the $\Hen$-definable set $\widetilde X$, as follows. 
For each maximal subset $C$ of $G_x$ with the property that $\ord ( a_j - a_{j'})\ge 0$ for all  $a_j,a_{j'}$ in $C$, let  $a_C$  be the average of the elements of $C$, and let $G'_x$ consist of the so-obtained elements $a_C$. Clearly, this condition depends definably on $x$, i.e., there exists an $\Hen$-definable set
$\widetilde G\subset X\times \VF$ such that $\widetilde G_x $ equals $G'_x$ for each point $x$ on $X$. Now consider an $\Hen$-definable injection
$\iota : \widetilde G \to X\times \RF^r$ for some $r\geq 0$ such that $\iota$ makes a commutative diagram with the projections to $X$ (such $\iota$ exists by cell decomposition \cite{Pas}, or, more simply, by a direct argument involving uniform finiteness and taking differences with  averages of close-by points). Let $\widetilde X$ be $\iota(\widetilde G)$ and write $p:\widetilde X \to X$ for the projection. We will now construct $\tilde f$ on $\widetilde X$ as desired. For each $i$, let $\widetilde p_i:Y_i\to \widetilde X$ be the $\cS$-definable function sending $y$ in $Y_i$ with $p_i(y)=x$ to $\iota(x,a_C)$, where $a_C$ is such that $h_i(y)$ belongs to the maximal subset $C$ of $G_x$ with the above property that $\ord ( a_j - a_{j'})\ge 0$ for all  $a_j,a_{j'}$ in $C$. For each $i$, let $\widetilde \xi_i :Y_i\to \RF$ be the $\cS$-definable function sending $y$ in $Y_i$ to $\xi_i(y) +  \ac (h_i(y)   -  a_C )$, where $\widetilde p_i(y) = \iota(x,a_c)$. Finally, let $h:\widetilde X\to\VF$ be the $\cS$-definable function sending $\iota(x,a_C)$ to $a_C$.
Now define $\tilde f$ as $E(h)c$, with $c= \sum_i p^*(H_i)\otimes  [\widetilde p_i:Y_i\to \widetilde X,\widetilde \xi_i,0]$.
This finishes the proof of the first part of the Proposition.

In statements (a) and (c), the implication ``$\Leftarrow$'' follows directly from $p_!(\tilde f) = f$.
The other direction
follows from the  inequalities from (\ref{eq:rel}) and by the definition of $\cCexp$ in Section 3 of \cite{CLexp} (which in particular has no relations in $\cCexp$ relating the motivic additive character evaluated in values which are far apart, that is, having differences of negative valuation). Note that (b) is a special case of (a), with $\cS(x)$ instead of $\cS$, and, with $\{x\}$ instead of $X$. This finishes the proof of the Proposition.
\end{proof}

The following two propositions give a criterion for integrability for integrals over valued fields variables by reducing to the situation of Proposition \ref{prop:crit}, of value group and residue field variables.
\begin{defn}[Boxes]\label{defn:box}
By a box $B$ in $\VF_K^n$ with $K$ in $\cS$ and $n\ge 0$ we mean a Cartesian product of $n$ balls  $\{x_i\mid \ord( x_i - a_i ) \ge  m_i\}$ for some $a_i\in \VF_K$ and some $m_i$ in $\VG_K$ for $i=1,\ldots,n$, with $B=\{0\}$ in case that $n=0$. Such a box $B$ has (valued field) dimension $n$, and, we write $\rad_n(B)$ for  $m_1+\dots+m_n$  (the sum of the valuative radii). %
\end{defn}
We follow \cite[Definition 7.9]{BGlockN} for the definition of strict $C^1$ functions. This notion
is a useful analogue for discrete valued fields of the notion of real $C^1$ functions. Indeed, see \cite{Glock,Bertram,BGlockN,Glock2006} for inverse and implicit function theorems for strict $C^1$  maps, and related results.\footnote{Note that in Definition 2.1.1 of \cite{CHLR}, which recalls the very same definition from \cite[Definition 7.9]{BGlockN}, the norm is missing in the numerator of the difference quotient.}

\begin{defn}[Strict $C^1$ functions]\label{defn:strict C1}
	A function $f:U\subset K^n\to K^m$ with $U$ open in $K^n$ and $K$ in $\cS$ is called strict $C^1$ 
at $a\in U$ if there is a matrix $A$ in $K^{m\times n}$ such that
		$$\lim_{(x,y)\to(a,a)} \ord (f(x) - f(y) - A\cdot(x-y)) - \ord (x-y) =+\infty $$
	where the limit is taken over $(x,y)\in U^2$ with $x\not=y$. Such $A$ is automatically unique and is denoted by $f'(a)$ or by $Df(a)$. 	The function $f$ is called strict $C^1$ if it is strict $C^1$  at each $a\in U$.
\end{defn}


\begin{prop}[Strict $C^1$ Parameterization on boxes]\label{prop:twisted:boxes}
Let $f_j$ be in $\cC(X)$, resp.~in $\cCe(X)$, for some $\cS$-definable sets $Z$ and  $X\subset \VF^n\times Z$ and for $j$ running over a finite set $J$. Then there exist $\cS$-definable subsets
$$
Y_i\subset \VF^i\times \VG^i\times \RF^m\times Z
$$
for $i=0,\ldots,n$ and some $m\ge 0$, and $\cS$-definable injections
$$
\varphi_i: Y_i\to X
$$
over $Z$ such that for each $i$, each $j$, and for each point $z$ on $Z_i:= p_i(Y_i)$ with $p_i:Y_i\to \VG^i\times \RF^m\times Z$ the projection map, the following hold.


\begin{itemize}
 \item[(a)]
The sets $X_i := \varphi_i(Y_i)$ are disjoint and their union equals $X$.

\item[(b)]  The fiber $Y_{i,z}= p_i^{-1}(z)$ is a box in $\VF^i$, and hence, is of (valued field) dimension $i$. 

\item[(c)] For each point $z$ on $Z_i$, the map $\varphi_{i,z}$ (being $\varphi_i$ restricted to $Y_{i,z}$) is a strict $C^1$ map which is moreover an isometry onto its image. Furthermore, the Jacobian of $\varphi_n$ is constant equal to $1$.

\item[(d)] There are $F_{ij}$ in $\cC(Z_i)$ (resp.~in $\cCe(Z_i)$) such that $\varphi_i^*(f_j)$ equals $p_i^*(F_{ij})$.
In particular, the restriction of  $\varphi_i^*(f_j)$ to $Y_{i,z}$ is constant.
\end{itemize}
\end{prop}
Recall that $\varphi_i$ being over $Z$ means that $\varphi_i$ makes a commutative diagram with the projections to $Z$. An isometry is for the supremum norm, as usual on non-archimedean fields,  with the valuation of a tuple being the minimum of the valuations of the tuple entries. 
\begin{proof}[Proof of Proposition \ref{prop:twisted:boxes}]
The proposition follows from an iterated application of the cell decomposition theorem of Pas \cite{Pas}, in the form of \cite[Theorems 7.2.1, 7.5.3]{CLoes}. (In Section 7.5 of \cite{CLoes} it is shown that the maps $\varphi_{i,z}$ can even be taken analytic.)
To see that only one part $Y_i$ is needed in each dimension $i=0,\ldots,n$, one notes that the corresponding disjoint unions can be easily realized by increasing $r$ if necessary. The isometric nature of the $\varphi_{i,z}$ follows by iteration in the $n$ variables of the variant of \cite[Theorem 3.2]{Pas} given by Theorem 7.5.3 of \cite{CLoes}. That the Jacobian of $\varphi_n$ can be taken to be identically equal to $1$ follows from  the Jacobian matrix being upper triangular with $1$ on the diagonal (indeed, the iterated cell decomposition gives  $\varphi_n$ as a triangular translation, like $(x,y)\mapsto (x,y+c(x))$ with $c$ being strict $C^1$).  The $F_{ij}$ are given by the (iterated) variant of Pas cell decomposition given by Theorem 7.2.1 of \cite{CLoes} in the case that the $f_j$ lie in $\cC(X)$. In the more general case that the $f_j$ lie in $\cCe(X)$, one writes $f$ as a finite sum of terms of the form $H_\ell\otimes [Y_\ell\to X,\xi_\ell,0]$ with $H_\ell$ in $\cC(X)$, and then the $F_{ij}$ are found as in the proof of Theorem 7.2.1 of \cite{CLoes} with the graphs of the $\xi_\ell$ instead of the sets $Y_\ell$ (that is, with the $a_i$ in  the proof of Theorem 7.2.1 of \cite{CLoes} being the classes of the graph of the $\xi_\ell$).
\end{proof}

\begin{prop}[Reduction of integrability to $\RF$ and $\VG$]\label{prop:VF-RF}
Let $f_j$, $X$, $Z$, $J$ be as in Proposition \ref{prop:twisted:boxes}, and choose the data $X_i$, $\varphi_i$, $Y_i$, $F_{ij}$, etc., as given by that proposition. Let  $d_j$ be integers with $0\le d_j\le n$ for  $j\in J$. For each $i$ and $j$, consider the $\cS$-definable function $g_i$ on $Z_i$ sending a point $z$ on $Z_i$ to $\rad_{i} (p_i^{-1}(z)))$ as in Definition \ref{defn:box}.
Then the following are equivalent.
\begin{itemize}
 \item[(i)] $f_j$ is integrable in relative dimension $d_j$ over $Z$ along the projection $X\to Z$.

\item[(ii)] 
One has
$F_{ij}=0$ when $i>d_j$, and, for $i=d_j$, the function $F_{ij} \LL^{-g_i}$ is integrable in relative dimension $0$ over $Z$ along the projection to $Z_i\to Z$.
 \end{itemize}
\end{prop}
The condition in (i) that $f_j$ in $\cC(X)$ is integrable in relative dimension $d_j$ over $Z$, means two conditions: Firstly, there exists an $\cS$-definable subset $D$ of $X$ of relative dimension $\le d_j$ over $Z$ such that multiplying $f_j$ with the characteristic function of $D$ yields $f_j$ again, and, secondly, the image of $f_j$ in $C^{d_j}(X\to Z)$ lies in $\mathrm{I}_ZC(X\to Z)$ with notation from \cite[Section 14]{CLoes}. For $\cCe$ instead of $\cC$ one uses  \cite[Section 4.3]{CLexp} instead of \cite[Section 14]{CLoes}.

\begin{proof}[Proof of Poposition \ref{prop:VF-RF}]
This follows at once for $\cC$ from an iterated applications of A7 and A8 of Theorems   10.1.1, 14.1.1 of \cite{CLoes}, and, correspondingly for $\cCe$ from Theorems 4.1.1, 4.3.1 of \cite{CLexp} (in the simple case of $\cCe$).
\end{proof}



\section{Proofs of the main results}\label{sec:proofs}
All definitions and auxiliary results for Theorems \ref{thm:no:null:intro} and \ref{thm:integrability-intro} 
have now been developed, so that we can proceed with their proofs.

By taking the difference of $f$ and $g$, it is enough to prove Theorem \ref{thm:no:null:intro} with $g=0$.
We can now jointly prove Theorem \ref{thm:no:null:intro} for $f$ in $\cC(X)$ and in $\cCe(X)$, with $g=0$. The case of $f$ in $\cCexp(X)$ will be reduced to the case of $f$ in $\cCe(X)$ and $g=0$ by Proposition \ref{lem:exp-to-con}, and, it will be finished by Corollary \ref{cor:no:null:main} for $\cCe$.

\begin{proof}[Proof of Theorem \ref{thm:no:null:intro} for $\cC(X)$ and $\cCe(X)$]
Let  $f$ be in $\cC(X)$, resp.~in $\cCe(X)$, such that $f(x)=0$ for all points $x$ on $X$. We need to show that $f=0$ in $\cC(X)$, resp.~in $\cCe(X)$.

\bigskip

\textbf{Case 1.} $X\subset \RF^m$ for some $m\ge 0$.

\bigskip

Briefly, this case follows from
our assumptions on $\Hen$ from Section \ref{sec:recall} (which in particular impose that the residue fields form an elementary class) and logical compactness.  One unwinds relations coming from localisation, from tensor product, and from the Grothendieck ring relations in the residue field part of the definition of $\cC(X)$, resp.~of $\cCe(X)$, to find relations in the free group generated by definable subsets in the residue field sort up to definable bijections.
Logical compactness then allows us get from a statement about all points on $X$ to a statement about the entire $X$. Let us develop this argument in detail.

Let us first treat the case that $f$ lies in $\cC(X)$.
The ring $\cC(X)$ is the localisation of $\cQ(X)$
by the multiplicative system generated by $\LL$ and $(1-\LL^i)$ for integers $i\not=0$, with notation from \Cref{sec:C} and where $\LL$ is the class of $\RF\times X$ (the class of the affine residue field line over $X$) in $\cQ(X)$. Thus, by multiplying $f$ by some power of $\LL$ and some factors of the form  $(1-\LL^i)$ for some nonzero $i$, we may (and will) suppose that $f$ is a difference of the form
$[A] - [B]$
for some $\Hen$-definable sets  $A\subset X\times \RF^{n_1}$ and $B\subset X\times \RF^{n_2}$ for some $n_1,n_2\geq 0$. Indeed, since those factors are units, this modification of $f$ preserves both (\ref{item:1:0}) and (\ref{item:2:0}) from the theorem.

For each point $x$ on $X$, the equation $f(x)=0$ implies, that there is an $\Hen(x)$-definable bijection $\psi_x$ between
the $\Hen(x)$-definable sets
\begin{equation}\label{eq.Ax}
\big(A_{x}\times  \RF^a\times \prod_{i\in I}(\RF^i \setminus \{0\}^i)^{a_i}\big) \sqcup W_x
\end{equation}
and
\begin{equation}\label{eq.Bx}
\big(B_{x} \times \RF^a\times \prod_{i\in I}(\RF^i\setminus \{0\}^i)^{a_i}\big) \sqcup W_x
\end{equation}
for some finite set of positive integers $I$, some integers $a,a_i\geq 0$, and some $\Hen(x)$-definable subset $W_x$ of $\RF^{m'}$ for some $m'\geq 0$, and where $\sqcup$ stands for the disjoint union.
(The expressions \eqref{eq.Ax} and \eqref{eq.Bx} arise from definition of the ring $\cC(X)$:
$\RF^a\times \prod_{i\in I}(\RF^i \setminus \{0\}^i)^{a_i}$ comes from the localization, and $W_x$ arises when passing from the semiring to the ring.)
Note that these data of $\psi_x$, $I$, $a$, $a_i$, $m'$, and $W_x$ depend on $x$. By elimination of valued field quantifiers from \cite{Pas}, these data only involve the residue field sort, with the ring language enriched with some constant symbols. Hence, by logical compactness and the fact that the residue fields of structures in $\Hen$ form an elementary class, we may suppose that $I$, $a$, $a_i$ and $m'$ do not depend on $x$ and that $W_x$ and $\psi_x$ are $\Hen$-definable uniformly in $x$. Hence, by invertibility of $\LL^a \cdot \prod_i(1-\LL^i)^{a_i}$, we find $f=0$. This finishes Case 1 for $f$ in $\cC(X)$. The argument for $f$ in $\cCe(X)$ is similar, using the relations for $\cQe(X)$ instead of $\cQ(X)$.

\bigskip

\textbf{Case 2.} $X \subset \VG^R\times \RF^n$ for some $R,n\ge 0$.

\bigskip

Apply Proposition \ref{lem:rec} to $X$ and $f$ (with $J=\{1\}$, and where we omit the index $j$ over $J$), to find $X_i$, $\theta_i$, and $L_i$ and associated data and properties as given by the Proposition. We may suppose that $X=X_i$ for some $i$.
Pulling back via $\theta_i$, we may suppose that $\theta_i$ is the identity map. Let us write $L$ for $L_i$ and $r$ for $r_i$. Thus, we have $X = A \times \NN^r$ for some $A\subset \Lambda\times \RF^n$ with $\Lambda$ finite, and $f = \sum_{(a,b)\in L}c_{a,b}\cdot g^a \cdot \LL^{g\cdot b}$, with $L \subset \NN^r \times \ZZ^r$ finite. That $\Lambda$ can be taken constant follows from (model theoretic) orthogonality between $\RF$ and $\VG$, which follows easily from Pas's quantifier elimination result in the form of (3.5) and (3.7) of \cite{vdDAx}, see also Theorem 2.1.1 of \cite{CLoes}.

Since the functions $\LL^{g_i}$ are invertible in $\cC(X)$, we may assume that all the exponents $b_{i}$ are non-negative, i.e., $L \subset \NN^{2r}$. (If not, multiply $f$ with suitable powers of the $\LL^{g_i}$; the theorem holds for the new $f$ if and only if it holds for the old $f$.)

We consider $f$ as a polynomial in $g_1$ and in $\LL^{g_1}$, i.e.,
\begin{equation}\label{eq:f}
f = P(g_1, \LL^{g_1})
\end{equation}
where $P \in \cR[x,y]$ is a polynomial with coefficients in the ring
$$\cR := \cC(A)[g_2, \dots, g_r, \LL^{g_2}, \dots, \LL^{g_r}]$$ (or with $\cCe(A)$ instead of $\cC(A)$).

Let $d = d(f)$ be the degree of $P$ in $y$ and let $e = e(f)$ be the degree of $P(x, 0)$ in $x$, where we define the degree of the zero-polynomial to be $-1$. We will proceed by induction on the numbers $r$, $d$, and $e$ (in lexicographical order).

The case $r=0$ is covered by Case 1. Indeed, $\Lambda$ can be sent into $\RF$ by an $\Hen$-definable injection. So, let us suppose that $r>0$.

If $e = -1$, then the function $f'$ obtained by dividing $f$ by $\LL^{g_1}$ also has the form (\ref{eq:f}), but with lower value $d(f')$. By induction, we obtain $f' = 0$, which implies $f = 0$. So from now on we suppose that $e \ge 0$ (which also implies $d \ge 0$).

We can also exclude the case $d = e = 0$, since in that case, $f$ does not depend on $g_1$, so that we can conclude using induction on $r$ (by writing $f$ as the pull-back of a function on $A \times \NN^{r-1}$).

\bigskip

For the remaining cases, define $f\new$ such that
$$f\new(z_1, \dots, z_r) = f(z_1+1, z_2, \dots, z_r) - f(z_1, \dots, z_r).$$
More concretely, $f\new = P\new(g_1, \LL^{g_1})$, where
\[
P\new(x,y) := P(x+1,\LL y) - P(x,y) \in \cR[x,y].
\]
Note that $f\new(x) = 0$ for all points $x$ on $X$.
Also note that $e(f\new) < e$ (the $x^d$-monomial cancels),
so by induction, $f\new$ is zero in $\cC(X)$. By the `furthermore' statement of Proposition \ref{lem:rec},
this implies that $P\new$ is the zero-polynomial.

\bigskip

Let us now first treat the case that $d=0$. (Recall that the case $d = e = 0$ has already been treated, so that we may now assume $e \ge 1$.)
As an intermediate step towards proving $f = 0$, let us derive that the polynomial $e\cdot P \in \cR[x,y]$ is zero. (Note that $\cR$ may have torsion.) Let $c \in \cR$ be the coefficient of the monomial $x^e$ in the polynomial $P$. Then the coefficient of $x^{e-1}$ in $P\new$ is equal to $e \cdot c$.
Since this is equal to zero, we obtain that $e\cdot P(x,0)$ has degree at most $e-1$ in $x$, so that we can apply induction to $e \cdot f$.
This shows that $e \cdot f = 0$ and hence, by the `furthermore' statement of Proposition \ref{lem:rec}, that $e \cdot P = 0$, as claimed.

Now we show, still for the case $d=0$ and $e>0$, that $f=0$.
To this end, for each $i=0,\ldots,e-1$, let $s_i$ be the $\cS$-definable map $A\times \NN^r\to A\times\NN^r $ sending $(a,z_1,\ldots,z_r)$ to $(a,i+ez_1,z_2,\ldots,z_r)$, and set $f_i := s_i^{*}(f)$. Then clearly $f_i$ evaluates to $0$ at every point on $X$, and it is of the form (\ref{eq:f}), that is, $f_i = P_i(g_1)$, with $P_i(x) = P(i + ex)$. (Recall that we are in the case $d = 0$, so that $P$ does not depend on $y$.) Using that $e \cdot P = 0$, we obtain $P_i(x) = P(i)$, namely as follows: Writing $P$ as $P = \sum a_\ell x^\ell$ (with $a_\ell \in \cR$), we have $e\cdot a_\ell = 0$ for each $\ell$, so when multiplying out $\sum a_\ell (i + ex^\ell)^\ell$, everything vanishes except $\sum a_\ell i^\ell$.

Hence, we find, by the case $d = e = 0$, that $f_i=0$ for all $i$. Using that $f$ can be expressed in terms of the $f_i$, we deduce that $f = 0$. Indeed, $f$ is the sum (over $i$) of the functions obtained by pulling $f_i$ back using the inverse of $s^*_i$ (which is defined on the image of $s_i$), and then extending this pull-back by $0$ on the complement of the image of $s_i$.

\bigskip

Finally, we treat the case $d > 0$. Let
$c \cdot x^{e'}y^d$ be the monomial of $P$ of maximal degree in $x$, among those of degree $d$ in $y$ (for some $c \in \cR \setminus \{0\}$). Then the coefficient of $x^{e'}y^d$ in $P\new$ is $(\LL^d - 1)\cdot c$. So $(\LL^d - 1)\cdot c = 0$, and since $(\LL^{d}-1)$ is invertible, we deduce $c = 0$, contradicting the choice of $e'$.

This finishes the proof of Case 2.

\bigskip

\textbf{Case 3.}  General case.

\bigskip

In this case, $X$ is a general $\cS$-definable set. One has  $X\subset \VF^n\times \RF^m\times \VG^R$ for some $n,m,R\ge 0$.  By Proposition \ref{prop:twisted:boxes} applied to $f_1=f$ and with $J=\{1\}$ and $Z=\RF^m\times \VG^R$, we obtain finitely many function $F_{i1}$ from item (d) of Proposition \ref{prop:twisted:boxes}. By the properties provided by that proposition, it is now enough to treat the case that $f$ is in fact equal to one of the $F_{i1}$. We may thus suppose that $n=0$, meaning that  $X\subset \RF^m\times \VG^R$ for some $m$ and $R$. (Alternatively, use Theorem 7.2.1(2) of \cite{CLoes} for such a reduction.) But this now falls under the scope of Case 2. The proof of Theorem \ref{thm:no:null:intro} for $\cC(X)$ and for $\cCe(X)$ is finished.
\end{proof}

Note that by the above proofs, we may from now on use Corollary \ref{cor:no:null:main} in the cases of $\cC$ and $\cCe$.

\begin{proof}[Proof of Theorem \ref{thm:no:null:intro} for $\cCexp(X)$]
Let $f$ be in $\cCexp(X)$ such that $f(x)=0$ for each point $x$ on $X$. We need to show that $f=0$.
We will use Proposition 
\ref{lem:exp-to-con} and Corollary \ref{cor:no:null:main} to reduce to the case that $f$ lies in $\cCe(X)$, for which we already proved Theorem \ref{thm:no:null:intro}.
Let $p:\widetilde X\to X$, $\tilde f \in \cCexp(\widetilde X)$, and $c \in \cCe(\widetilde X)$ be as given by Proposition \ref{lem:exp-to-con}.
By our assumption that $f(x)=0$ for each point $x$ on $X$, we find that $\tilde f(\tilde x)=0$ for each $\tilde x$ in $\widetilde X$ by  (b) of Proposition \ref{lem:exp-to-con}. Hence, by the first part of Proposition \ref{prop:crit} we find that $c(\tilde x)=0$ for each $\tilde x$ in $\widetilde X$.    
Hence, by Corollary \ref{cor:no:null:main} (for the case $\cCe$) we find that $c=0$.
This clearly implies that $\tilde f=0$ (again by the first part of Proposition \ref{prop:crit}), which in turn implies $f = \mu_{X,0}(\tilde f)=0$, as desired, with $\mu_{X,0}$ as in \Cref{lem:exp-to-con}. This finishes the proof of Theorem \ref{thm:no:null:intro}. \end{proof}

We end the paper with the proof of Theorem \ref{thm:integrability-intro}.

\begin{proof}[Proof of Theorem \ref{thm:integrability-intro}]   
By the definition of integrability over $Z$ along $g$, statement (1) of the theorem implies statement (2). Let us prove the reverse implication. So, we are given $f$ and $g$ such that for each point $z$ on $Z$, the restriction $f_{|g^{-1}(z)}$ is integrable in dimension $d$.

\bigskip

\textbf{Case 1.} $X\subset \VG^{R} \times \RF^n \times Z$ for some $R\ge 0$ and some $n\ge 0$, $g:X\to Z$ is the projection, and $d=0$.

\bigskip

Apply Proposition~\ref{lem:rec} to $f_1=f$ with $J=\{1\}$ to obtain maps $\theta_i$ and corresponding objects $L_i$, $r_i$, $c_{a,b}$, etc., as in the proposition, where we omit the index $j \in J$.
By finite additivity we may suppose that $X=X_i$ for some $i$. 
By pulling back under $\theta_i$ we may in fact suppose that $\theta_i$ is the identity map; in other words, without loss, $f$ itself is of the form \eqref{f:sum:int:prop} given in Proposition~\ref{lem:rec}.
We are given that for any point $z$ in $Z$, the restriction $f_{|g^{-1}(z)}$ is integrable. Thus, Proposition \ref{prop:crit} (applied to this restriction) implies that $c_{a,b|g^{-1}(z)}=0$ for each point $z$ on $Z$ and each $(a,b) \notin \NN^{r_i} \times (\ZZ \setminus \NN)^{r_i}$. Using Corollary \ref{cor:no:null:main}, one obtains that $c_{a,b}=0$ for those $(a,b)$. Hence, $f$ is integrable in relative dimension $0$  over $Z$ along $g$, by Proposition \ref{prop:crit} (or just by the definition of integrability). This finishes Case 1.
\\

\textbf{Case 2.} General case.    \\

In this case $X$ is a general $\cS$-definable set.
The case where $f$ lies in $\cC(X)$ or $\cCe(X)$ reduces immediately to Case~1
by Propositions \ref{prop:twisted:boxes} and \ref{prop:VF-RF} (see Case~3 of the proof of Theorem \ref{thm:no:null:intro}).

Let us finally treat the case that $f$ is in $\cCexp(X)$. We reduce to the previous case of $\cCe$ instead of $\cCexp$, again by Proposition \ref{lem:exp-to-con}, as follows.
Suppose that for each point $z$ on $Z$, the restriction $f_{|g^{-1}(z)}$ is integrable in dimension $d$.
Let $\widetilde X$, $p$, $c$, and $\tilde f$ be given by Proposition \ref{lem:exp-to-con}.
By (c) of Proposition \ref{lem:exp-to-con}, it is enough to prove that  $\tilde f$ is integrable in relative dimension $d$ over $Z$ along $g\circ p$.  By (c) of Proposition \ref{lem:exp-to-con}, now applied with  $Z$ replaced by $\{z\}$ for a point  $z$ on $Z$ and $f$ replaced by $f_{|g^{-1}(z)}$,  it follows that  $\tilde f_{|(g\circ p)^{-1}(z)}$ is integrable in dimension $d$. By the equivalence of (i) with (iv) from Proposition \ref{prop:crit}, it follows that $c_{|(g\circ p)^{-1}(z)}$ is integrable in dimension $d$ for each point $z$ on $Z$.  Hence, by the already proved part of Case~2 for $\cCe$, applied to $c$, we find that $c$ is integrable in relative dimension $d$ over $Z$ along $g\circ p$.  Again by the equivalence of (i) with (iv) from Proposition \ref{prop:crit}, we find that $\tilde f$ is integrable in relative dimension $d$ over $Z$ along $g\circ p$.
This finishes the proof of Theorem \ref{thm:integrability-intro}.
\end{proof}

\begin{remark}
As mentioned, our results and proofs have natural adaptations to the related frameworks of  Section 3.1 of \cite{CLbounded} and the corresponding ones from \cite{Kien:rational}, since the key points behind Theorems \ref{thm:no:null:intro} and \ref{thm:integrability-intro} are the facts that the residue fields form an elementary class, that the residue field is orthogonal (in the model theoretic sense) to the value group  which carries the pure Presburger structure, and, the reduction to situations without valued field variables essentially enabled by cell decomposition \cite{Pas}.
More general (future) adaptations to the more powerful axiomatic framework of hensel-minimality of \cite{CHR} is possible under the corresponding conditions, which come partly for free under hensel minimality.
\end{remark}

\bibliographystyle{amsplain}
\bibliography{anbib}

\end{document}